\documentclass[]{amsart}
\usepackage[]{amsmath, amsthm, amsfonts, amssymb}
\usepackage[]{graphicx}
 \usepackage[all]{xy}

 \newcommand {\Star} {{\mathrm Star}}

 \newcommand {\conv} {{\mathrm conv}}
 
 \newcommand {\C} {{\mathbb C}}
 \newcommand {\R} {{\mathbb R}}
 \newcommand {\Z} {{\mathbb Z}}
 \newcommand {\Q} {{\mathbb Q}}
\newcommand{\dt} {{\bullet}}
\newcommand {\dd} {{\partial}}
\newcommand {\db} {{\bar \partial}}
\newcommand {\F} {{\mathcal F}}
\newcommand {\OO} {{\mathcal O}}
\newcommand {\X}  {{\mathcal X}}

\newcommand {\PP} {\mathbb{P}}
\newcommand {\HH} {\mathbb{H}}
\newcommand {\cI} {\mathcal{I}}
\DeclareMathOperator*{\im}{im}
\DeclareMathOperator*{\coker}{coker}

\newtheorem{thm}[subsection]{Theorem}
 \newtheorem{cor}[subsection]{Corollary}
 \newtheorem{lemma}[subsection]{Lemma}
 \newtheorem{prop}[subsection]{Proposition}
 
 \newtheorem{rmk}[subsection]{Remark}
 \newtheorem{ex}[subsection]{Example}
  
 \newtheorem{conj}[subsection]{Conjecture}
\newtheorem {quest}[subsection]{Question}

\begin{document}
\title[Weights on cohomology]{ Weights on cohomology,  invariants of singularities, and dual complexes}
\author{Donu Arapura}
\address{D. A. \& J. W.: Department of Mathematics\\
   Purdue University\\
   West Lafayette, IN 47907\\
   U.S.A.}
\author{Parsa Bakhtary}
\address{P. B.: Department of Mathematics and Statistics\\
   King Fahd University of Petroleum and Minerals\\
   Dhahran, Saudi Arabia 31261}
\author{Jaros{\l}aw W{\l}odarczyk}
\thanks{The first author is partially supported by the NSF,  the third  supported
  in part by an NSF grant and a Polish KBN grant}
 
\date{\today}
\maketitle

\begin{abstract}
  We study the  weight filtration on the cohomology
  of a proper complex algebraic variety and  obtain  natural upper
  bounds on its size, when it is the exceptional divisor of a
  singularity. We also give bounds for the cohomology of links.
  The  invariants of singularities introduced  here gives rather  
strong information about the topology of 
rational and related singularities.
\end{abstract}

Given a divisor on a variety,  the  combinatorics governing  the way the
components intersect is encoded by the associated dual complex. This  is the  simplicial complex
  with $p$-simplices corresponding  to $(p+1)$-fold
  intersections of components of the divisor.
Kontsevich and Soibelman \cite[A.4]{ks} and 
Stepanov \cite{step2} had  independently observed that the homotopy type of the
 dual complex of a simple normal crossing exceptional divisor
associated to a resolution of an  isolated singularity is an invariant
for the singularity. In fact,  \cite{ks},  and later Thuillier \cite{th} and 
Payne \cite{payne} have obtained homotopy  invariance results for more general dual complexes,
such as those arising from boundary divisors.  
In characteristic zero, all these results are  consequences of  the weak factorization theorem of W\l
odarczyk [Wlo], and Abramovich-Karu-Matsuki-W\l odarczyk
[AKMW]; Thuillier uses
rather different methods based on  Berkovich's  non-Archimedean analytic geometry.
As we show here, a slight refinement of factorization (theorems \ref{fact}, \ref{fact2}) and of these
techniques  yields some generalizations   this statement.
This applies to divisors of resolutions of arbitrary not
 necessarily isolated singularities, and even in a more general
 context  (discussed in the final section).
We also allow dual complexes associated with nondivisorial  varieties,
such as fibres of resolutions of nonisolated singularites. 
Here is a slightly imprecise formulation of theorems \ref{8} and \ref{10}.

\begin{thm}\label{thm:intro}
Suppose that  $X$ is a smooth complete variety. Let
$E\subset X$ be a  divisor with simple normal crossings, or more generally
a union of smooth subvarieties  which is local analytically a union of
intersections of coordinate hyperplanes.
 Then the homotopy type of the
dual complex of $E$  depends only the  the
complement $X-E$, and in fact only on its proper birational class.
\end{thm}

This theorem and  related ones  in the final section were inspired by
the work of  Payne \cite[\S 5]{payne},  Stepanov \cite{step2},  and Thuillier \cite{th} 
mentioned above (and  in  hindsight also by \cite{ks}, although we
were unaware of this paper at the time these results were completed).

To  put the remaining results in context, we note that
the present paper can be considered as the extended version of our earlier
preprint \cite{abw}, where  we gave the bounds on the cohomology of
the dual  complex of singularities (not necessarily isolated). As we
will explain below, these  results are established in a more  refined
form in the present paper.  To explain the motivation, we recall the
following from the end of \cite{step2}:

 \begin{quest} [Stepanov]\label{conj:1} Is the dual complex   associated to the
   exceptional divisor of  a good resolution of an isolated
rational singularity  contractible?
\end{quest}

This  seems to have been motivated by a result of Artin \cite{artin}, that
the exceptional divisor of a resolution of a rational surface
singularity consists of a tree of rational curves. However, it appears
to be somewhat  overoptimistic.
Payne \cite[ex. 8.3]{payne} has found  a counterexample, where the dual complex
is homeomorphic to $\R\PP^2$.  Nevertheless, a weaker form of this
question have a positive answer. Namely that the higher Betti numbers of the  dual
complex associated to a resolution of an isolated rational
vanish. As we have recently learned, this was first observed by Ishii
\cite[prop 3.2]{ishii} two decades  ago. Here we give a
 stronger statement and place it in a more general context. The key step is to study  this problem from the much more
general perspective of Deligne's  weight filtration \cite{deligne}.  The point is
that for   a simple
 normal crossing divisor D,
  the weight zero  part exactly coincides with the cohomology  of the
  dual complex $\Sigma_D$:

 \begin{lemma}[Deligne]  $W_0(H^i(D,\C))=H^i(\Sigma_D, \C)$
\end{lemma}

 Because of this, $W_0$ was
referred to as the ``combinatorial part'' in \cite{abw}.
The inclusion in the lemma is induced by the map collapsing the components of $D$ to the
vertices of $\Sigma_D$ (figure 1). This is explained in more detail in
section 2.

\begin{center}
\begin{figure}[ht]
    \centering
      \includegraphics[height=2.5in]{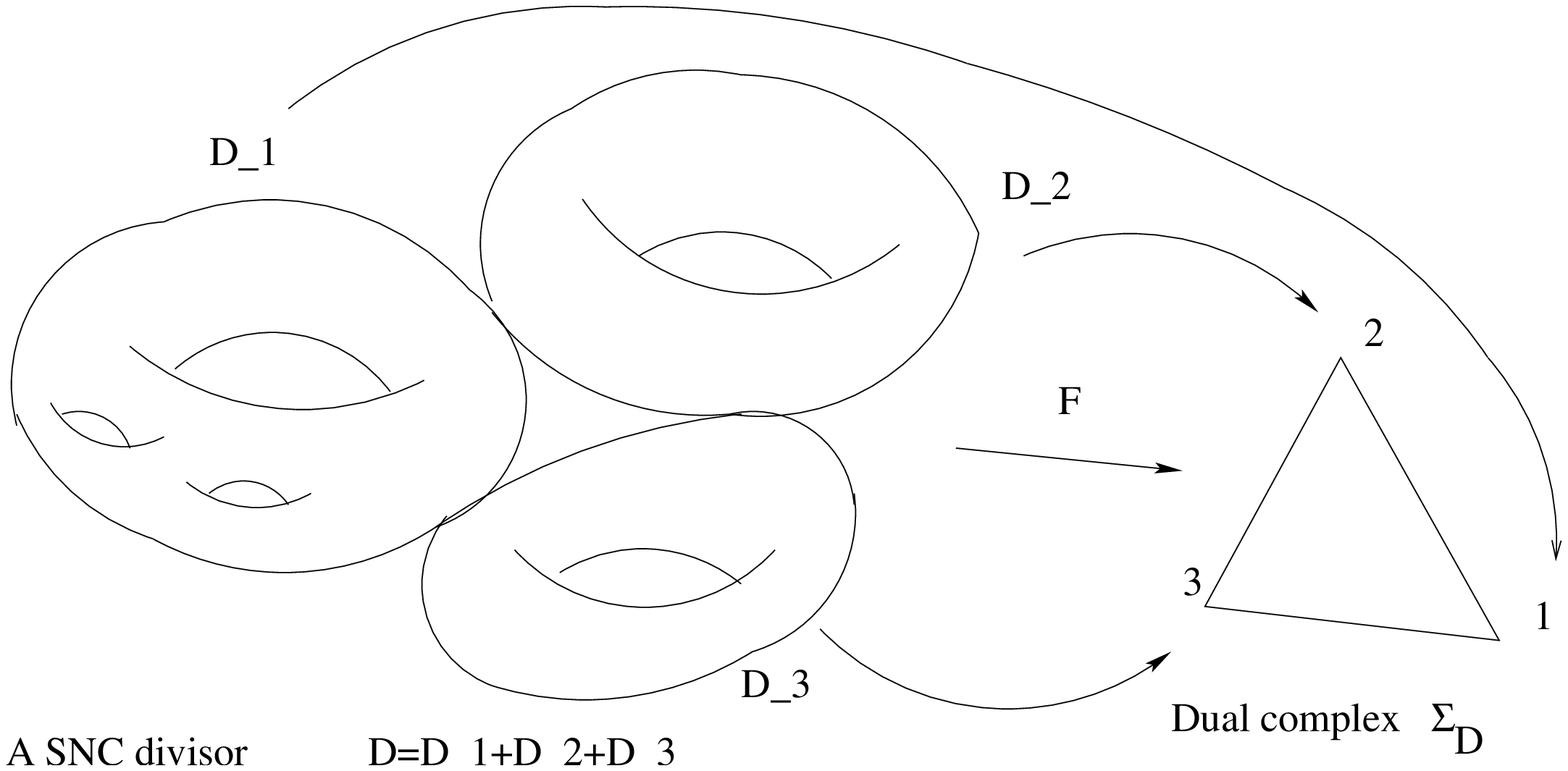}
 \caption{}
  \end{figure}
\end{center}

The above result can be applied to the exceptional fibre of any
resolution $f:X\to Y$ of any (not necessarily isolated) singularity, leading to the invariant
$$\dim(W_0(H^i(f^{-1}(y),\C))=\dim(H^i(\Sigma_{f^{-1}(y)}, \C))$$
More generally, this suggests the study of $w^i_j(y)=\dim(W_j(H^i(f^{-1}(y)))$ and also the
weight spaces of the intersection cohomology of the link.  This is the
principal goal of our paper. 
In precise terms,
given a proper map of varieties $f:X\to Y$, we give a bound on $\dim
W_jH^i(f^{-1}(y))$ in terms of the sum of $(i-p)$th cohomology of $p$-forms along the
 fibres for $p\le j$.  More concrete results can be extracted for appropriate
classes of singularities. Specializing to $w^i_0(y)$, we
find that it is bounded above by $\dim (R^if_*\OO_X)_y\otimes
\OO_y/m_y$. This refines \cite[prop 3.2]{ishii}.
The number $w_0^i(y)$ also vanishes for $0<i<\dim Y-1$ for a resolution of an
isolated normal Cohen-Macaulay singularity $(Y,y)$.
This was initially observed in \cite{abw}. Payne \cite{payne} has given this
an interesting reinterpretation as saying that  the rational homotopy type of
dual complex  of a Cohen-Macaulay singularity
is wedge of spheres. When $(Y,y)$ is toroidal, we show that
$w_j^i(y)=0$ for $j<i/2$ and $i>0$.
We have similar bounds on  $\dim W_jIH^i(L)$, where $L$
is the link of $y$. These are reduced to the previous results by using
the decomposition theorem.

The paper is organized as follows. In the first two sections we briefly
introduce the elements of the theory of Deligne's weight filtration. In
Sections 3 and 4,  we discuss the bounds for $W_0$, and
$W_j$. In Section 5 we introduce and study the invariants of
singularities describing the topology of fibres of resolution. In
Section 6 we study the invariants related to topology of the links of
singularities. In Section 7 we prove theorems on boundary divisors.
We note that in this paper the term ``scheme'', without further
 qualification, refers to a  scheme of  finite type over $\C$. A
 variety is a reduced scheme.

\subsection{Acknowledgements} The third author  heartily thanks the Max Planck Institut f\"ur 
Mathematik, Bonn for its warm hospitality. Our thanks also to the referee
for bringing \cite{ks} to our attention, and  to J\'anos Koll\'ar
for some other comments.

\section{ Properties of Deligne's Weight filtration}

Although the weight filtration is part of a much   more elaborate
story -- the theory of mixed Hodge structures -- it seems useful present the elementary
features independently of this.
A simple normal crossing divisor  provides the key  motivating example
for the construction of   the  weight filtration. If $X$ is a union
of two components, the cohomology fits into  the Mayer-Vietoris
sequence. In general, it is computed by a Mayer-Vietoris (or
\v{C}ech or descent) spectral sequence 
$$E_1^{pq}= \bigoplus_{i_0<\ldots <i_p} H^q(X^{i_0\ldots i_p})\Rightarrow H^{p+q}(X)$$
where $X^{i_0i_1\ldots}$ are intersections of components.
The associated filtration on $H^*(X)$ is precisely the weight filtration
$W_0H^i(X)\subseteq W_1H^i(X)\subseteq \ldots$. In explicit terms
\begin{equation}
  \label{eq:Wsnc}
W_jH^i(X) = \im H^i(X, \ldots 0\to \bigoplus_{m_0<m_1\ldots }
\pi_*\Q_{X^{m_0\ldots m_{i-j}}}\to\ldots )  
\end{equation}
via the resolution
$$0\to \Q_X\to \bigoplus_m \pi_*\Q_{X^m}\to \bigoplus_{m<n}
\pi_*\Q_{X^{mn}}\to $$
where $\pi:X^{mn\ldots}\to X$ denote the inclusions.

More generally, we introduce the weight filtration in an axiomatic way.
Given an complex algebraic variety or scheme $X$, Deligne
\cite{deligne}  defined
the weight filtration $W_\dt H_c^i(X)$ on the
compactly supported rational cohomology. It is an increasing
filtration possessing  the following
properties:
\begin{enumerate}
\item[(W1)] These subspaces are preserved by proper pullbacks.

\item[(W2)] If $X$ is a divisor with simple normal crossings in a
  smooth complete variety, then
  $W_\dt H_c^i(X)=W_\dt H^i(X)$ is the filtration associated to the Mayer-Vietoris
  spectral sequence in \eqref{eq:Wsnc}. In particular, if
$X$ is smooth and complete, $W_jH^i(X)= 0$ for $j<i$ and $W_jH^i(X) = H^i(X)$
for $j\geq i$.

\item[(W3)] If  $U\subseteq X$ is an open immersion and $Z=X-U$, then
the standard exact sequence
$$\ldots H_c^{i-1}(Z)\to H_c^i(U)\to H_c^i(X)\to \ldots$$
restricts to an exact sequence
$$\ldots W_jH_c^{i-1}(Z)\to W_jH_c^i(U)\to W_jH_c^i(X)\to \ldots$$

\item[(W4)] Weights are multiplicative:
$$W_jH_c^i(X\times Y) = \bigoplus_{a+b=j,r+s=i} W_aH_c^r(X)\otimes W_bH_c^s(Y)$$

\end{enumerate}

The goal of the remainder of this section is to prove that the axioms
given above uniquely characterize the weight filtration on smooth or
projective varieties. In fact, we will not need the last axiom for this.
The proof of existence will be reviewed in the
next section.

\begin{lemma}\label{lemma:smooth}
  Suppose that $W_jH^i_c(-)$ are  $W'_jH_c^i(-)$ are two families of filtrations satisfying  (W1)-(W3).
 Then $W'_jH_c^i(X)=W_jH_c^i(X)$ for all smooth $X$.
\end{lemma}

\begin{proof}
  By resolution of singularities \cite{Hir}, we may choose a smooth compactification
  $\bar X$ of $X$ so that $E=\bar X-X$ is divisor with simple normal
  crossings. Then from the exact sequence 
$$\ldots W_jH_c^{i-1}(\bar X) \to W_jH_c^{i-1}(E)\to W_jH_c^i(X)\to
W_jH_c^i(\bar X)\to \ldots$$
we deduce that
$$W_jH_c^i(X)=
\begin{cases}
  H_c^i(X) &\text{ if $j\ge i$}\\
\coker[H^{i-1}(\bar X)\to H^{i-1}(E)] & \text{ if $j=i-1$}\\
W_jH^{i-1}(E) & \text{ if $j<i-1$}
\end{cases}
$$
The filtration $W'$ would have an identical description.
\end{proof}

\begin{lemma}\label{lemma:mayer}
 Assume that $W_jH_c^i(X)$ satisfies the  axioms (W1)-(W3).
 Given a   variety $X$
with a closed set $S$ and a desingularization $f:\tilde X\to X$
which is an isomorphism over $X-S$. Let $E=f^{-1}(S)$.
 Then there is an exact sequence
$$
\ldots \to W_jH_c^{i-1}(E)\to W_jH_c^i(X)\to
W_jH_c^{i}(\tilde X)\oplus W_jH_c^{i}(S)\to \ldots
$$

\end{lemma}

\begin{proof}
 This   follows from a diagram chase on
$$
\xymatrix{
\ldots W_jH_c^{i-1}(S)\ar[r]\ar[d] & W_jH_c^i(U)\ar[r]\ar[d]^{=} & W_jH_c^i(X)\ar[r]\ar[d] & W_jH_c^i(S)\ar[d]\ldots \\ 
 W_jH_c^{i-1}(E)\ar[r] & W_jH_c^i(U)\ar[r] & W_jH_c^i(\tilde X)\ar[r] & W_jH^i(E)
}
$$
where $U= X-S$.
\end{proof}

\begin{prop}\label{prop:proj}
 Suppose that $W_jH^i_c(-)$ are  $W'_jH_c^i(-)$ are two families of filtrations satisfying  (W1)-(W3).
Then $W'_jH_c^i(X)=W_jH_c^i(X)$ for all projective $X$.
\end{prop}

\begin{proof}
  The proof is inspired by the work of El Zein \cite{elzein}. Choose
  an embedding $X\subset P=\PP^n$. Let $\pi:\tilde P\to P$ be an
  embedded resolution of singularities such that $E=\pi^{-1}(X)$ is a
  divisor with simple normal crossings. By lemma~\ref{lemma:mayer}, we
  have an exact sequence
$$
\ldots \to W_jH^i(P)\to
W_jH^{i}(\tilde P)\oplus W_jH_c^{i}(X)\to W_jH^{i}(E) \to W_jH^{i+1}(P)\ldots
$$
This implies that
$$
W_jH_c^i(X)=
\begin{cases}
  H^i_c(X) & \text{ if $j\ge i$}\\
W_jH^i(E) &\text{ if $j<i$}
\end{cases}
$$
and this determines $W_j$ uniquely.
\end{proof}

As an illustration of this method, we can recover an elementary
description of the weight filtration on smooth or projective toric
varieties due to Weber \cite{weber}. Fix an
integer $p>1$ (not necessarily prime). The map given by multiplication
by $p$ on tori extends to a ``Frobenius-like'' endomorphism
$\phi_p:X\to X$ for any toric variety $X$ [loc. cit.]
 When $X$ is smooth and projective, $\phi_p$ acts on $H^i(X)$ by
 $p^{i/2}$. In general,  we define
$W_jH^i_c(X)$ to be the sum of eigenspaces of $\phi_p^*$ with
eigenvalue $p^k$ with $k\le j/2$.  This filtration can be checked to
satisfy (W1)-(W3) restricted to the category of toric varieties. 

\begin{prop}[Weber] If $X$ is a toric variety,
  then $W$ coincides with Deligne's weight filtration.
\end{prop}

When $X$ is smooth or projective, this can be proved exactly as in
lemma~\ref{lemma:smooth} and proposition~\ref{prop:proj} with the
additional (realizable) constraint that the intermediate varieties and maps
$X\subset \bar X$, $X\subset P$ and $\pi:\tilde P\to P$ are
constructed within the category of toric varieties.

\section{Simplicial resolutions}

The general construction is based on simplicial resolutions.
Perhaps, a motivating example is in order. Given a divisor $X$ with
simple normal crossings, the underlying combinatorics of how the components
fit together is determined by  the dual complex $\Sigma_X$. This  is the simplicial
complex having one vertex for each connected component; a simplex
lies in $\Sigma_X$ if and only the corresponding components meet.
The cohomology of the  dual complex is precisely $W_0H^i(X)$. In order
to describe  the rest of weight
filtration in this fashion, we need the full simplicial resolution. Before describing it,
we recall some standard material \cite{deligne, gnpp, ps}.
A simplicial object in a category is a diagram
$$
\xymatrix{ \ldots X_2\ar[r]\ar@<1ex>[r]\ar@<-1ex>[r] &
  X_1\ar@<1ex>[r]\ar@<-1ex>[r] & X_0 }
$$
with $n$ face maps $\delta_i:X_n\to X_{n-1}$ satisfying the standard
relation $\delta_i\delta_j = \delta_{j-1}\delta_i$ for $i<j$; this
would be more accurately called a ``strict simplicial'' or
``semisimplicial'' object since we do not insist on degeneracy maps
going backwards.  The basic example of a simplicial set,
i.e. simplicial object in the category of sets, is given by taking
$X_n$ to be the set of $n$-simplices of a simplicial complex on an
ordered set of vertices. Let $\Delta^n$ be the standard $n$-simplex
with faces $\delta_i':\Delta^{n-1}\to \Delta^n$.  Given a simplicial
set or more generally a simplicial topological space, we can glue the
$X_n\times \Delta^n$ together by identifying $(\delta_i x,y)\sim
(x,\delta_i' y)$. This leads to a topological space $|X_\dt|$ called
the geometric realization, which generalizes the usual construction of
the topological space associated a simplicial complex. 
Here is a basic example.

\begin{ex}\label{ex:ncd1}
  Suppose that $X$ is a topological space given as union of open or closed subsets 
  $X^i$. Let 
  $X^{ij\ldots}=X^i\cap X^j\ldots$ denote the intersections.  Then a 
  simplicial space is given by taking $X_n$ to be the disjoint
  union of  the $(i+1)$-fold intersections $X^{i_0\ldots i_n}$. The face
  map $\delta_k$ is given by inclusions
$$X^{i_0\ldots i_n}\subset X^{i_0\ldots\hat i_k\ldots i_n}\>(
i_1<\ldots< i_k)$$ 
When $X$ is triangulable, and $X^i$ are subcomplexes, then $|X_\dt|$
and $X$ are homotopy equivalent. 
\end{ex}

\begin{ex}\label{ex:ncd}
  The above construction and comment applies to the case of an analytic space 
$X$ with irreducible components $X^i$. Of particular interest is the
case where $X$ is a divisor with simple normal crossings; in this case
$X_\dt$ is referred to as the canonical simplicial resolution of
$X$.  Applying the connected
component functor results in a simplicial complex $\pi_0(X_\dt)$ which
none other than the dual complex $\Sigma_X$.
\end{ex}

The notion  of a  simplicial resolution can be extended to arbitrary
varieties as follows:

\begin{thm}[Deligne]
 Given any (possibly reducible) variety $X$, there
exists a smooth simplicial variety $X_\dt$, which 
we call a {\em simplicial resolution},
with a collection of proper morphisms $\pi_\dt:X_\dt\to X$ (commuting
with face maps), called an augmentation, inducing
a homotopy equivalence between $|X_\dt|$ and $X$. Given a morphism $f:X\to Y$
there exists simplicial resolutions $X_\dt, Y_\dt$ and a 
morphism $f_\dt:X_\dt\to Y_\dt$ compatible with $f$.
\end{thm}

 The theorem is a consequence of resolution of singularities. Proofs can be found 
in \cite{deligne,gnpp, ps}.  Note that the original construction of
Deligne results in a  necessarily infinite diagram, 
whereas the method of Guillen, Navarro Aznar et. al
yields a fairly economical resolution. The canonical simplicial
resolution of normal crossing divisor is a simplicial resolution in
this technical sense.

\begin{ex}\label{ex:isolated}
  Following the method of \cite{gnpp} we can construct a simplicial
  resolution of a variety $X$ with isolated singularities as
  follows. Let $f:Y\to X$ be a resolution of singularities such
  that the exceptional divisor $E = \cup E^i$ is a divisor with simple
  normal crossings. Write $E^{ij}= E^i\cap E^j$ and $E_{n} = \coprod
  E^{i_0\ldots i_n}$.  Let $S_0\subset X$ be the set of singular points,
  $S_1\subseteq S_0$ be the set of images of $\cup E^{ij}$ and so
  on. Then the simplicial resolution is given by
$$
\xymatrix{ \ldots E_1 \sqcup S_2\ar[r]\ar@<1ex>[r]\ar@<-1ex>[r] & E_0
  \sqcup S_1\ar@<1ex>[r]\ar@<-1ex>[r] & Y\sqcup S_0 }
$$
where the face maps are given by inclusions $S_i\to S_{i-1}$ on the
second component. On the first component $\delta_k$ is given by
$$
\begin{cases}
  E^{i_1\ldots i_n}\subset E^{i_i\ldots\hat i_k\ldots i_n} &\text{if $k\le n$}\\
  f: E^{i_1\ldots i_n}\to S_{n-1} &\text{if $k=n+1$}
\end{cases}
$$
\end{ex}

 Figure 2 depicts the simplest example of this, which is the  simplicial resolution 
$$
\xymatrix{  X_1=E\ar@<1ex>[r]\ar@<-1ex>[r] & X_0=Y\sqcup x }
$$
with its augmentation to $X$, for
an isolated singularity $x$ resolved by a single blow up. Figure 3
depicts the geometric realization  of this simplicial resolution,
which is the cone over $E$ in $Y$. This is homotopic to $X$.

\begin{center}
\begin{figure}[ht]
    \centering
     \includegraphics[height=3.0in]{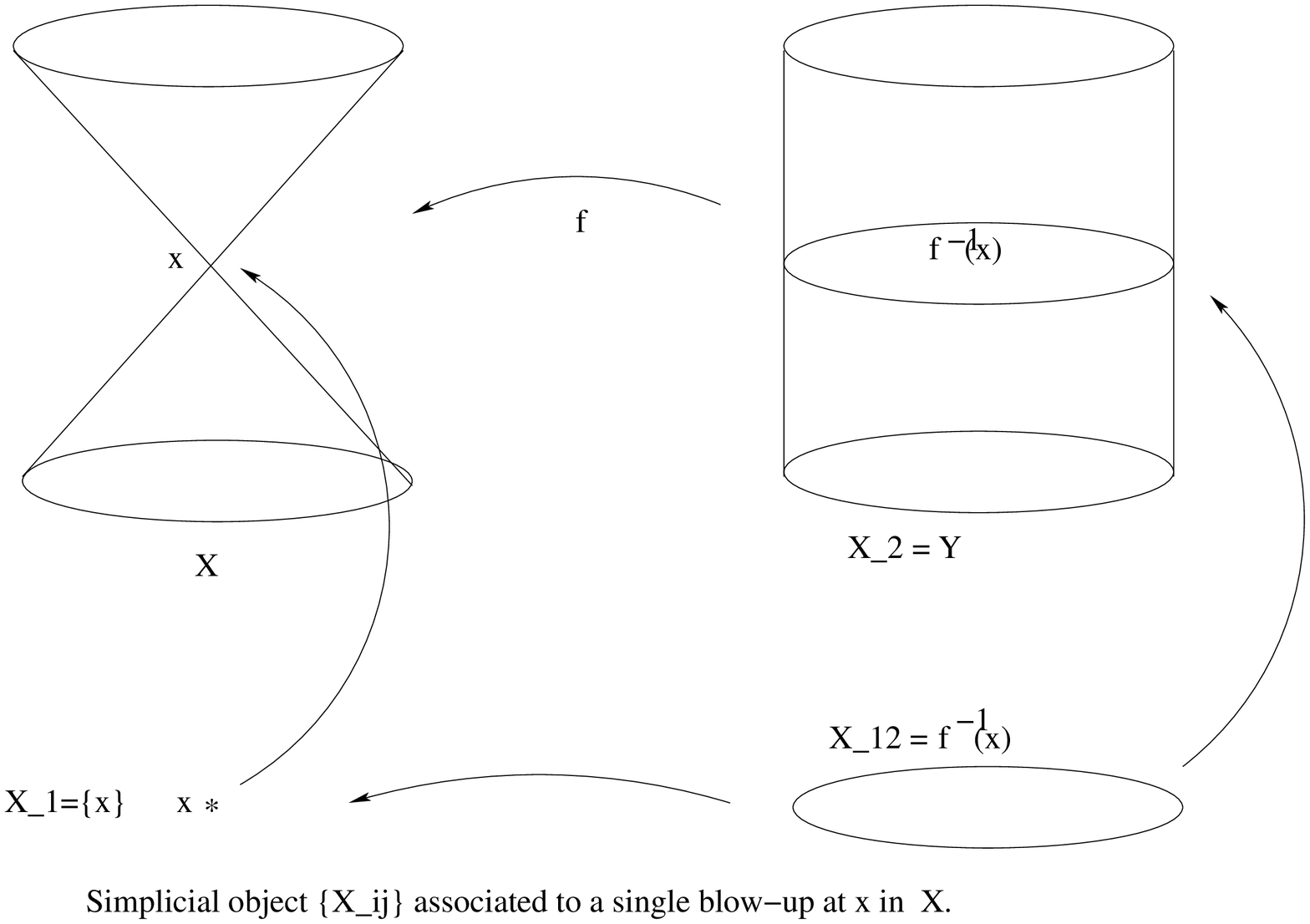}
 \caption{}
  \end{figure}
\end{center}

\begin{center}
\begin{figure}[ht]
    \centering
     \includegraphics[height=3.5in]{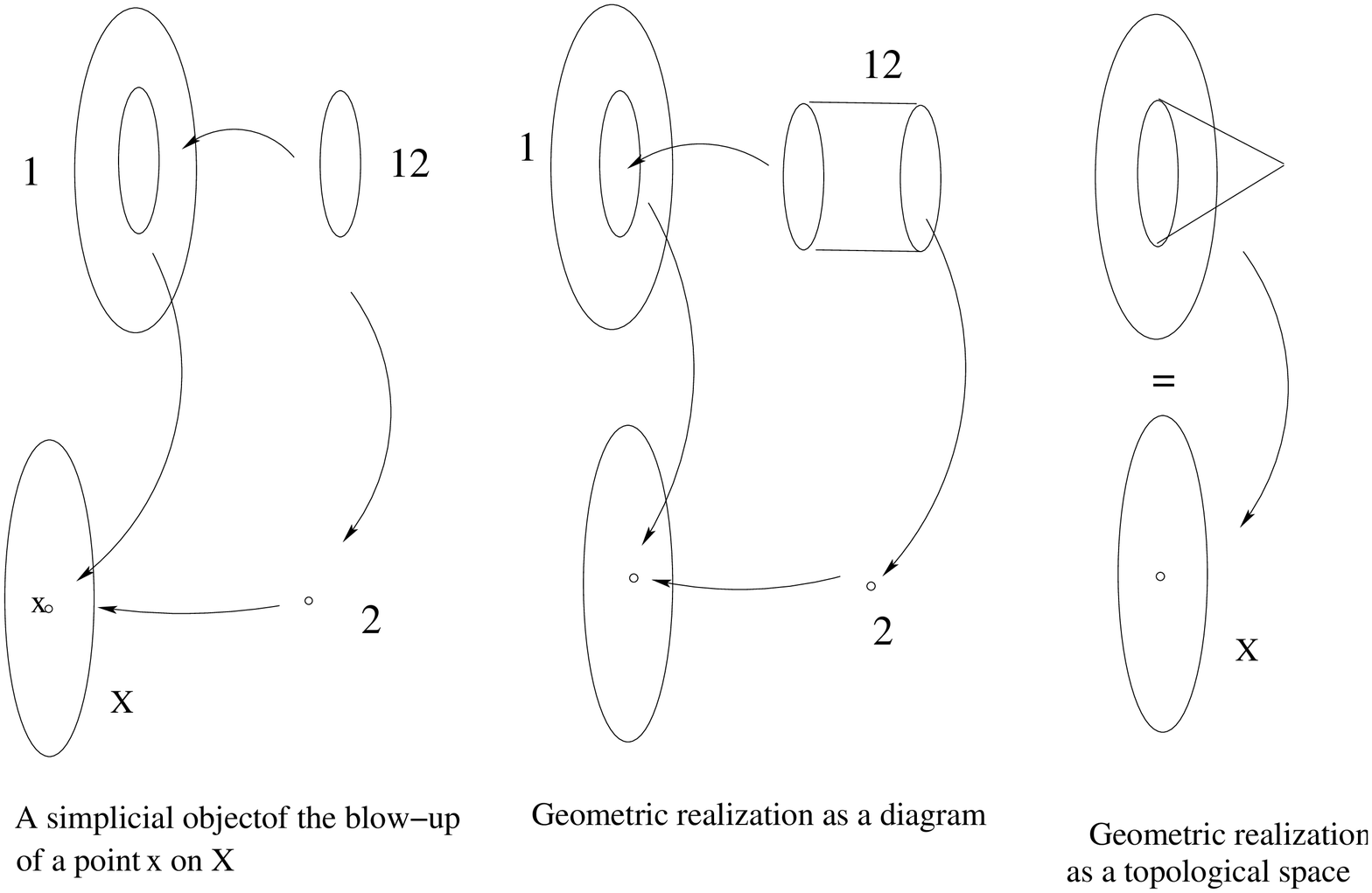}
 \caption{}
  \end{figure}
\end{center}

Given a simplicial space, filtering $|X_\dt|$ by skeleta
$\bigcup_{n\le N} X_n\times \Delta^n/\sim$ yields the spectral
sequence
\begin{equation}
  \label{eq:ss1}
  E_1^{pq} = H^q(X_p,A)\Rightarrow H^{p+q}(|X_\dt|,A)  
\end{equation}
for any abelian group $A$. Part of the datum of this spectral sequence
is the filtration on $H^*(|X_\dt|,A)$ induced by skeleta.
When applied to the canonical simplicial resolution of a divisor $X$ with
simple normal crossings, we recover the Mayer-Vietoris spectral
sequence given earlier.

It is convenient to extend this. A simplicial sheaf on $X_\dt$ is a
collection of sheaves $\F_n$ on $X_n$ with ``coface'' maps
$\delta_i^{-1}\F_{n-1}\to \F_n$ satisfying the face relations.  

\begin{ex}
  The constant sheaves $\Z_{X_\dt}$ with identities for coface
maps forms a simplicial sheaf. 
\end{ex}

\begin{ex}
  Suppose that $\pi_\dt:X_\dt\to X$ is a morphism of spaces commuting
  with face maps. Then the pullback of any
  sheaf $\F_\dt=\pi^{-1}\F$ is naturally a simplicial sheaf. The
  previous example is of this form, but not all are.
\end{ex}

\begin{ex}
If $X_\dt$ is a simplicial object in
the category of complex manifolds, then $\Omega^i_{X_\dt}$ with the
obvious maps, forms a simplicial sheaf.  This example is not of the
previous form.
\end{ex}

We can define cohomology by
setting
$$H^i(X_\dt,\F_\dt) = Ext^i(\Z_{X_\dt},\F_\dt)$$
This generalizes sheaf cohomology in the usual sense, and it can be
extended to the case where $\F_\dt^\dt$ is a bounded below complex of
simplicial sheaves by using a hyper $Ext$.  When $\F=A$ is constant,
this coincides with $H^{i}(|X_\dt|,A)$. But in general the meaning is
more elusive.  There is a spectral sequence
\begin{equation}
  \label{eq:ss2}
  E_1^{pq}(\F^\dt_\dt) = H^q(X_p,\F^\dt_p)\Rightarrow H^{p+q}(X_\dt, \F^\dt_\dt)  
\end{equation}
generalizing \eqref{eq:ss1}.  To be explicit, the differentials are
given by the alternating sum of the  compositions
$$H^q(X_p,\F_p^\dt)\to H^q(X_{p+1},\delta_i^{-1}\F_p^\dt)\to H^q(X_{p+1},\F_{p+1}^\dt)$$
Filtering $\F^\dt$ by the ``stupid
filtration'' $\F_\dt^{\ge n}$ yields a different spectral sequence
\begin{equation}
  \label{eq:ss3}
  {}'E_1^{pq} = H^q(X_\dt,\F^p_\dt)\Rightarrow H^{p+q}(X_\dt, \F^\dt_\dt)  
\end{equation}

\begin{thm}[Deligne]\label{thm:degen}
  If $X_\dt$ is a simplicial object in the category of compact
  K\"ahler manifolds and holomorphic maps. The spectral sequence
  \eqref{eq:ss1} degenerates at $E_2$ when $A=\Q$. 
\end{thm}

\begin{rmk}
  The theorem follows from a more general result in
  \cite[8.1.9]{deligne}. However the argument is very complicated.
  Fortunately, as pointed out in \cite{dgms}, this special case
  follows easily from the $\dd\db$-lemma.  Here we give a more
  complete argument.
\end{rmk}

\begin{proof}
  It is enough to prove this after tensoring with $\C$.  We can
  realize the spectral sequence as coming from the double complex
  $(E^\dt(X_\dt),d,\pm\delta)$, where $(E^\dt,d)$ is the $C^\infty$ de
  Rham complex, and $\delta$ is the combinatorial differential. (We
  are mostly going to ignore sign issues since they are not relevant
  here.) In fact this is a triple complex, since each $E^\dt(-)$ is
  the total complex of the double complex $(E^{\dt\dt}(-),\dd,\db)$.

  Given a class $[\alpha]\in H^i(X_j)$ lying in the kernel of
  $\delta$, we have $\delta\alpha = d\beta$ for some $\beta\in
  E^{i-1}(X_{j+1})$ Then $d_2([\alpha])$ is represented by
  $\delta\beta\in E^{i-1}(X_{j+2})$.  We will show this vanishes in
  cohomology. The ambiguity in the choice of $\beta$ will turn out to be 
  the key point.

  By the Hodge decomposition, we can assume that $\alpha$ is pure of
  type $(p,q)$. Therefore $\delta\alpha$ is also pure of this type. We
  can now apply the $\dd\db$-lemma \cite[p 149]{gh} to write $\alpha = \dd\db \gamma$
  where $\gamma\in E^{p-1,q-1}(X_{j+1})$. This means we have two
  choices for $\beta$.  Taking $\beta=\db\gamma$ shows that
  $d_2([\alpha])$ is represented by a form of pure type $(p-1,q)$. On
  the other hand, taking $\beta = -\dd\gamma$ shows that this class is
  of type $(p,q-1)$.  Thus $d_2([\alpha]) \in H^{p-1,q}\cap
  H^{p,q-1}=0$.

  By what we just proved $\delta\alpha = d\beta, \delta\beta=d\eta$,
  and $\delta\eta$ represents $d_3([\alpha])$. It should be clear that
  one can kill this and higher differentials in the exact same way.
 
\end{proof}

\begin{cor}\label{cor:degen}
  With the same assumptions as the theorem, the spectral sequence
  \eqref{eq:ss2} degenerates at $E_2$ when $\F=\OO_{X_\dt}$.
\end{cor}

\begin{proof}
  By the Hodge theorem, the spectral sequence for $\F=\OO_{X_\dt}$ is
  a direct summand of the spectral sequence for $\F = \C$.
\end{proof}

 We are, at last, in a position to explain the {\em construction   of the weight
  filtration} for  a complete  variety $X$. Choose a  simplicial
resolution $X_\dt\to X$ as above.  Then
the spectral sequence (\ref{eq:ss1}) will then converge to
$H^*(X,\Q)$ when $A=\Q$.  The weight filtration $W$ is the induced increasing
filtration on $H^*(X)$ indexed so that
$$ W_qH^{p+q}(X)/W_{q-1} = E_\infty^{pq}$$
Although $X_\dt$ is far from unique, Deligne \cite{deligne} shows that
$W$ is well defined, and moreover that this  part of the datum of the canonical  mixed
structure.  By theorem~\ref{thm:degen}, we obtain

\begin{lemma}\label{lemma:weight}
We have $W_{-1}=0$ and
  $$W_jH^i(X,\Q)/W_{j-1}  \cong H^{i-j}(\ldots\to H^j(X_p,\Q)\to H^j(X_{p+1},\Q)\ldots )$$
\end{lemma}

\begin{cor}\label{cor:weight}
  $W_jH^i(X)=H^i(X)$ if $j\ge i$ and $W_jH^i(X)=0$ if $i-j>\dim X$.
\end{cor}

\begin{proof}
  The first part is an  immediate consequence of the lemma. For
  the second,  we observe that the work of Guillen, Navarro Aznar
et. al. \cite{gnpp},\cite[thm 5.26]{ps}, shows that a  simplicial
resolution $X_\dt$ can be chosen with length at most $\dim X$.
So that $H^{i-j}(H^j(X_\dt))$ is necessarily $0$ for $i-j>\dim X$.
\end{proof}

Applying $\pi_0$ to $X_\dt$ results in a simplicial set. We have a
canonical map $X_\dt\to \pi_0(X_\dt)$ of simplicial spaces which
induces a continuous map of 
$$X\sim |X_\dt|\to |\pi_0(X_\dt)|$$
which is well defined up to homotopy.

\begin{cor} \label{inclusion}
The map on cohomology is injective and
  $$W_0H^i(X,\Q)  =\im H^i(|\pi_0(X_\dt)|,\Q)$$
\end{cor}

\begin{rmk}
When $X$ is a divisor with simple normal crossings, this says that $W_0H^i(X)$ is the
cohomology of the dual complex.  In this case the
inclusion $$H^i(|\Sigma_X|,\Q)\to H^i(X,\Q)$$ \noindent  can be
constructed more directly. It is induced by a simple, but less
canonical simplicial map $\phi: X\to |\Sigma_X|$ described in
Stepanov \cite[lemma 3.2]{step}. Take the triangulation of $X$ such
that components $X_i$ and their intersections $$X_{i_1,i_2,...,i_k} :
= X_{i_1}\cap X_{i_2}\cap \ldots  \cap X_{i_k}$$\noindent
are simpilicial subcomplexes. Denote by $\Delta_{i_1,i_2,...,i_k}$
the simplex in the dual complex $\Sigma_X$ which corresponds to
$X_{i_1,i_2,\ldots,i_k}$.  Then we make a barycentric subdivisions
$\overline{\Sigma}$ of ${\Sigma}$ and  $\overline{\Sigma}_X$ of
${\Sigma}_X$. For any vertex $v$ of  $\overline{\Sigma}$ which lies
in the minimal component $X_{i_1,i_2,\ldots,i_k}$ we
put  $$\phi(v):=\mbox{\rm  the\quad  center \quad of \quad  the \quad
  simplex}\quad  \Delta_{i_1,i_2,...,i_k}$$
See figure 4.
This construction is homotopically equivalent to the construction in Corollary \ref{inclusion}. In particular, it
has connected fibres.
\end{rmk}

\begin{center}
\begin{figure}[ht]
    \centering
     \includegraphics[height=2.5in]{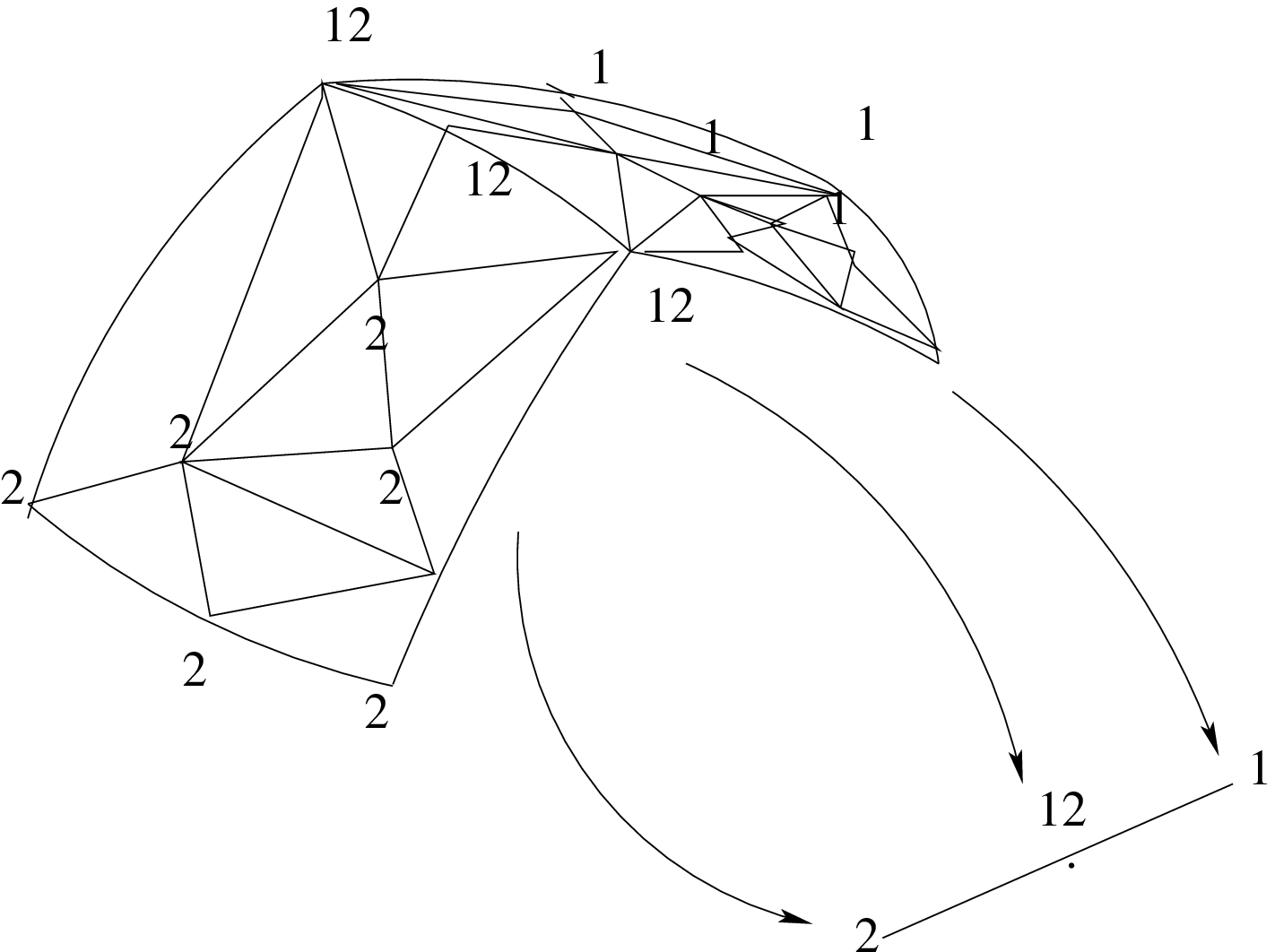}
 \caption{}
  \end{figure}
\end{center}

We can now describe the construction of the weight filtration for an
arbitrary variety $U$. Choose a compactification $X$. Denote the
complement by $\iota:Z\subset X$. There exists
 simplicial resolutions $Z_\dt \to Z$, $X_\dt\to X$ and a morphism $\iota_\dt:Z_\dt\to X_\dt$ covering  $\iota$. Then there is a new smooth simplicial variety $cone(\iota_\dt)$ (\cite[\S 6.3]{deligne}, \cite[IV \S 1.7]{gnpp}) whose geometric realization is 
homotopy equivalent to $X/Z$. We have a spectral sequence
 $$E_1^{pq}= H_c^q(X_p-Z_p, \Q)\Rightarrow H^*_c(X-Z,\Q).$$
 The weight filtration $W$ is defined via this spectral sequence as
 above. Deligne
\cite{deligne} shows that conditions (W1), (W2),  (W3) and (W4)
are satisfied.

For any variety $X$, we can construct a simplicial variety $X_\dt$
with $|X_\dt|$ homotopic to $X$  as in
example \ref{ex:ncd}. It is not a simplicial resolution in general, but
it is dominated by one. If we apply
$\pi_0$ to this simplicial variety, we get a simplicial set $\Sigma_X$
canonically attached to $X$, that we will call the nerve or dual
complex.  There is a canonical map $H^i(|\Sigma_X|,\Q)\to H^i(X,\Q)$
coming from the spectral sequence \eqref{eq:ss1} associated to this
simplicial variety. From the above discussion, we can see that:

\begin{lemma}
  If $X$ is complete, the image $H^i(|\Sigma_X|,\Q)\to H^i(X,\Q)$ lies in $W_0H^i(X,\Q)$.
  If $X$ satisfies the assumptions of example \ref{ex:ncd}, then these
  subspaces coincide.
\end{lemma}

\begin{lemma}
  Let $\pi:\tilde X\to X$ be a resolution of a complete variety
  such that the exceptional divisor $E$ has normal crossings. Let $S=\pi(E)\subset X$. Then
  $\dim W_0H^i(X)$ is the $(i-1)$st Betti number $b_{i-1}$ of the dual
  complex of $E$  when $i> 2\dim(S)+1$.  If $S$ is
  nonsingular, then this holds for $i>1$.  When $i=2\dim(S)+1$, $\dim
  W_0H^i(X) = b_{i-1}$ minus the number of irreducible components of
  $S$ of maximum dimension.
\end{lemma}

\begin{proof}
  This follows from lemma~\ref{lemma:mayer}, 
and the above remarks. 
\end{proof}

\section{Bounds on $W_0$ of a fibre.}

Suppose that $X$ is a complete variety i.e., proper reduced scheme. Then in addition to the weight
filtration,
$H^i(X,\C)$ carries a second filtration, called the Hodge filtration
induced on the abutment $H^i(X,\Omega_{X_\dt}^\dt)\cong H^i(X,\C)$ of
the spectral sequence (\ref{eq:ss3}) for $\Omega_{X_\dt}^\dt$ for a
simplicial resolution $\pi_{\dt}:X_\dt\to X$. By
convention $F$ is decreasing. We have $F^0= H^i(X,\C)$ and
$$F^0H^i(X,\C)/F^{1} \cong H^i(X_\dt,\OO_{X_\dt})$$

The filtration $W$ induces the same filtration on $ H^i(X_\dt,\OO_{X_\dt})$
as the one coming from (\ref{eq:ss2}). In particular,
\begin{equation}
  \label{eq:1W0F0}
  \begin{split}
W_0 Gr^0_FH^i(X,\C) &=  H^i(\ldots\to H^0(X_p,\OO)\to H^0(X_{p+1},\OO)\ldots )\\
&=H^i(\ldots\to H^0(X_p,\C)\to H^0(X_{p+1},\C)\ldots )\\
&\cong W_0H^i(X,\C)
\end{split}
\end{equation}
This means that Hodge filtrations becomes trivial on $W_0H^i(X)$. So
that this is a vector space and nothing more.

This leads to one of the main theorems of this paper.

\begin{thm}\label{thm:main}
  \begin{enumerate}
  \item[]
  \item[(a)] Suppose that $X$ is a proper (not necessarily reduced) scheme, then there is a canonical inclusion
    $$W_0H^i(X,\C)\hookrightarrow H^i(X,\OO_X),$$
\noindent which is a restriction of the natural map
  $\kappa:H^i(X,\C)\rightarrow H^i(X,\OO_X)$ induced by the morphism of sheaves $\C_X\to \OO_X$.

  \item[(b)] Suppose that  $f:X\to Y$ a proper morphism of varieties. Then there is an
    inclusion $W_0H^i(f^{-1}(y),\C)\hookrightarrow
    (R^if_*\OO_X)_y\otimes \OO_y/m_y$ for each $y\in Y$.
  \end{enumerate}
\end{thm}

\begin{proof}
By \eqref{eq:1W0F0},
$$W_0H^i(X,\C)= W_0 Gr^0_FH^i(X,\C)$$
thus $W_0H^i(X,\C)$ injects into $Gr^0_FH(X,\C)$ under the canonical
map $H^i(X,\C)\to Gr^0_F H^i(X,\C)$.
  Since there is a factorization
$$
\xymatrix{
  H^i(X,\C)\ar^{\kappa}[rr]\ar[rd] &   &H^i(X,\OO_{X}) \ar[ld]\\
   &  Gr^0_FH^i(X,\C) & }
$$
the  restriction of $\kappa$ to $W_0H^i(X,\C)$ is also necessarily
injective. To be clear, we are factoring this as
$$H^i(X,\C)\to H^i(X,\OO_X)\to H^i(X_{red},\OO_{X_{red}})\to Gr_F^0H^i(X,\C)$$

For (b), let $X_y$ be the reduced fibre over $y$, and $X_y^{(n)}$ the fibre
with its $n$th infinitesimal structure. From (a), we have a
natural inclusion $s:W_0H^i(X_y,\C)\hookrightarrow
H^i(X_y,\OO_{X_y})$.  After choosing a simplicial resolution of the fibre
$f_\dt:\X_{\dt}\to X_y$, $s$ can be identified with the composition
$$E^{i0}_2(\C)\to E^{i0}_2(\OO_{\X_\dt})\to H^i(X_y,\OO_{X_y})$$
where the first map is induced by the natural map $\C\to \OO$, and the last
map is the edge homomorphism.  Applying the same construction to
the simplicial sheaf $f_\dt^*\OO_{X_y^{(n)}}$ yields a map $s_n$
fitting into a commutative diagram
$$
\xymatrix{
  W_0H^i(X,\C)\ar^{s}[rr]\ar^{s_n}[rd] &  & H^i(X_y,\OO_{X_y})\\
  & H^i(X_y,\OO_{X_y^{(n)}})\ar[ru] & }
$$
Furthermore, these maps are compatible, thus they assemble  into a map $s_\infty$ to the
limit. Together with the formal functions theorem \cite[III
11.1]{hartshorne}, this yields a commutative diagram
$$
\xymatrix{
  W_0H^i(X,\C)\ar^{s}[rr]\ar^{s_\infty}[d]\ar^{s'}[rrd] &  & H^i(X_y,\OO_{X_y})\\
  \varprojlim H^i(X_y,\OO_{X_y^{(n)}})\ar^{\sim}[r] &
  (R^if_*\OO_X)_y\hat{}\ar[r] & (R^if_*\OO_X)_y\otimes \OO_y/m_y\ar[u]
}
$$
Since $s$ is injective, the map labeled $s'$ is injective as well.
\end{proof}

\begin{cor}\label{cor:main}
  Suppose that $f:X\to Y$ is a resolution of singularities.
  \begin{enumerate}
  \item If $Y$ has rational singularities then
    $W_0H^i(f^{-1}(y),\C)=0$ for $i>0$.
  \item If $Y$ has isolated normal Cohen-Macaulay singularities,
    $W_0H^i(f^{-1}(y),\C)=0$ for $ 0< i < \dim Y-1$
  \end{enumerate}

\end{cor}

\begin{proof}
  There is no loss in assuming that the exceptional divisor has normal
  crossings. Then the first statement is an immediate consequence of the theorem.  The
  second follows from the well known fact given below. We sketch the
  proof for lack of a suitable reference.

\begin{prop}\label{prop:CM}
  If $f:X\to Y$ is a resolution of a variety with isolated normal
  Cohen-Macaulay singularities, then $R^if_*\OO_X = 0$ for $0<i<\dim
  Y-1$
\end{prop}

\begin{proof}[Sketch]
  We can assume that $Y$ is projective. By the Kawamata-Viehweg
  vanishing theorem \cite{kawamata,viehweg}
  \begin{equation}
    \label{eq:kV}
    H^i(X, f^*L^{-1})=0,\quad i < \dim Y = n, 
  \end{equation}
  where $L$ is ample.  Replace $L$ by $L^N$, with $N\gg 0$.  Then by
  Serre vanishing and Serre duality (we use the CM hypothesis here)
  \begin{equation}
    \label{eq:SV}
    H^i(Y, L^{-1}) = H^{n-i}(Y, \omega_Y\otimes L) = 0,\quad i < n.  
  \end{equation}
  The Leray spectral sequence together with~\eqref{eq:kV} and
  \eqref{eq:SV} imply
$$ H^0(R^if_*\OO_X\otimes L^{-1}) = 0,\quad i< n-1$$
Since the sheaves $R^if_*\OO_X$ have zero dimensional support, the
proposition follows.
\end{proof}

\end{proof}

Corollary \ref{cor:main} refines \cite[prop 3.2]{ishii}.
Concerning the first conjecture, we have the following partial result
in the rational homotopy category.

\begin{cor}
The rational homotopy type of the dual complex associated to a resolution of an isolated rational hypersurface 
singularity of dimension $\geq 3$ is trivial.
\end{cor}

\begin{proof}
This follows from the corollary~\ref{cor:main} and  Stepanov's result that the dual complex associated to a resolution 
of an isolated hypersurface singularity of dimension $\geq 3$ is simply connected \cite{step}. 
\end{proof}

\section{Bounds on higher weights of a fibre.}

Theorem~\ref{thm:main} can be refined to get bounds on $W_j$ using ideas
of du Bois \cite{dubois, ps} which we recall below.
Given a simplicial resolution $\pi_\dt:X_\dt\to X$, 
we can construct the  derived direct image
$$\tilde\Omega_{X}^p= \R
\pi_{\dt*}\Omega^p_{X_\dt}$$
In more explicit terms, this can
realized by the total complex of
$$\pi_{0*}G^\dt(\Omega^p_{X_0})\to \pi_{1*}G^\dt(\Omega^p_{X_1})\to\ldots$$
where $G^\dt$ is Godement's flasque resolution. 
These fit together into a bigger complex $\tilde \Omega_X^\dt$ filtered by $p$
$$
 F^p \left\{
\begin{array}{cccc}
 \pi_{0*}G^\dt(\Omega^p_{X_0}) & \to & \pi_{1*}G^\dt(\Omega^p_{X_0}) & \to \\ 
 \downarrow &  & \downarrow &  \\ 
 \pi_{0*}G^\dt(\Omega^{p+1}_{X_0}) & \to & \pi_{1*}G^\dt(\Omega^{p+1}_{X_0}) & \to
\end{array}
\right.
$$
More precisely,
$$(\tilde \Omega_{X}^\dt,F^p)=\R
\pi_{\dt*}(\Omega^\dt_{X_\dt},\Omega^{\ge p}_{X_\dt})$$
in the filtered derived category.   

There is a natural map
$\Omega_X^p\to \tilde \Omega_X^p$ from the $p$th exterior power of the
sheaf of K\"ahler differentials. This is not a quasi-isomorphism in  general.
We summarize the basic  properties:

\begin{enumerate}
\item As objects in the (filtered) derived category 
$(\tilde \Omega_{X}^\dt,F)$ and $\tilde \Omega_X^p$ are  independent of the simplicial resolution.
\item  There is a  map $(\Omega_X^\dt,\Omega_X^{\ge \dt})\to (\tilde\Omega_X^\dt,F)$, from the
  complex of K\"ahler differentials, such that
  composing with $\C\to \Omega_X^\dt$ yields
 a quasi-isomorphism $\C_X\cong \tilde \Omega_X^\dt$.
\item The filtration $F$ on $\tilde \Omega_X^\dt$ induces the Hodge
  filtration on cohomology, and the associated spectral sequence
  degenerates at $E_1$ when $X$ is proper.

\end{enumerate}

From these statements, we extract
$$F^pH^i(X,\C)/F^{p+1} \cong \HH^i(X,\tilde\Omega_{X}^p)$$
when $X$ is proper. If $X\subseteq Z$ is a closed immersion, then
we get a map
$$\Omega_Z^\dt|_X\to \tilde \Omega_X^\dt$$
by composing the map in (2) with  restriction to $\Omega_X^\dt$.

\begin{thm}
  \begin{enumerate}
  \item[]
\item[(a)] Suppose that $X\subseteq Z$ is  a closed immersion of
  proper scheme into another scheme, then there is a
  canonical inclusion
$$H^i(X)/F^{j+1}\hookrightarrow \HH^i(X,\Omega_Z^{\le j}|_X)$$
 \item[(b)] Suppose that $X\subset Z$ is as in (a), then there is a canonical inclusion
$$\dim W_jH^i(X,\C)\hookrightarrow \HH^i(X,\Omega_Z^{\le j}|_X)$$
\item[(c)]  Suppose that  $f:X\to Y$ a proper morphism of varieties. Then there is an
    inclusion $W_jH^i(f^{-1}(y),\C)\hookrightarrow
    (\R^if_*\Omega_X^{\le j})_y\otimes \OO_y/m_y$ for each $y\in Y$
  \end{enumerate}
   
\end{thm}

\begin{proof}
By the remarks preceding the theorem, we have an 
 isomorphism
$$H^i(X)/F^{j+1}\cong \HH^i(X,\tilde  \Omega_X^{\le j})$$
For (a) it suffices to observe that this
factors through $\HH^i(X,\Omega_Z^{\le j}|_X)$. So the corresponding map
is injective.

By lemma~\ref{lemma:weight} ,
$$Gr^W_jH^i(X,\C)= H^{i-j}(\ldots\to H^j(X_k,\C)\to H^j(X_{k+1},\C)\ldots )$$
Therefore $F^{j+1}\cap W_jH^i(X)=0$. So that the natural map
$$W_jH^i(X,\C)\to H^i(X)/F^{j+1}$$
is injective. Composing this with the map in (a) yields an injection
$$s:W_jH^i(X,\C)\to \HH^i(X,\Omega_Z^{\le j }|_X)$$
 This proves (b).

The argument for  (c) is basically a reprise of the  proof of
theorem~\ref{thm:main} (b).
Let $X_y$ be the reduced fibre over $y$, and $X_y^{(n)}$ the fibre
with its $n$th infinitesimal structure. 
$$
\xymatrix{
  W_jH^i(X,\C)\ar^{s}[rr]\ar^{s_n}[rd] &  & \HH^i(X_y,\Omega_{X}^{\le j}|_{X_y})\\
  & \HH^i(X_y,\Omega_{X}^{\le j})|_{X_y^{(n)}}\ar[ru] & }
$$
Furthermore, these maps are compatible, thus they pass to map $s_\infty$ to the
limit. Then by the formal functions theorem, this yields a commutative diagram
$$
\xymatrix{
  W_jH^i(X,\C)\ar^{s}[rr]\ar^{s_\infty}[d]\ar^{s'}[rrd] &  &
  \HH^i(X_y,\Omega^{\le j}_{X_y})\\
  \varprojlim \HH^i(X_y,\Omega^{\le j}_{X_y^{(n)}})\ar^{\sim}[r] &
  (\R^if_*\Omega^{\le j}_X)_y\hat{}\ar[r] & (\R^if_*\Omega^{\le j}_X)_y\otimes \OO_y/m_y\ar[u]
}
$$
Since $s$ is injective, the map labeled $s'$ is injective as well.
\end{proof}

\begin{cor}\label{cor:2ndthm}
  \begin{enumerate}
  \item[]
\item With the same assumptions as in (a), we have
 $$\dim W_jH^i(X)\le \sum_{p\le j} \dim
 H^{i-p}(X,\Omega_Z^p|_X)$$
\item With the same assumptions as in (c), we have
 $$\dim W_jH^i(f^{-1}(y)) \le \sum_{p\le j} 
\dim  (R^{i-p}f_*\Omega_X^{p})_y\otimes \OO_y/m_y$$
  \end{enumerate}

\end{cor}

In special cases, we can use this to get rather precise bounds. We
need the following relative version  of Danilov's theorem
\cite[7.6]{danilov}.

\begin{prop}\label{prop:toricvan}
  Suppose that  $f:X\to Y$ is a projective  toric between toric varieties, with
  $X$ smooth. Then
$R^qf_*\Omega_X^p=0$ for $q>p$.
\end{prop}

\begin{proof}
The argument is similar to the proof of proposition \ref{prop:CM}.
We can first assume that $Y$ is affine and  then replace it with a
projective toric compactification. Let $L$ be an ample line bundle on
$Y$. After replacing $L$ by $L^N$, with $N\gg 0$, Serre's vanishing
implies that we can assume that
$R^qf_*\Omega_X^p\otimes L$ is globally generated and that
the Leray spectral sequence collapses to an isomorphism
$$H^q(X,\Omega_X^p\otimes f^*L)= H^0(R^qf_*\Omega_X^p\otimes L)$$
The left side vanishes by a theorem of Mavlyutov \cite{mavlyutov}, and
so the proposition follows.
\end{proof}

\begin{prop}\label{prop:toric}
  If $f:X\to Y$ is a projective  toric morphism between toric varieties, then
$W_jH^i(f^{-1}(y)) =0$ when $j<i/2$.
\end{prop}

\begin{proof}
  We have
$$\dim W_jH^i(f^{-1}(y)) \le \sum_{p\le j} 
\dim  (R^{i-p}f_*\Omega_X^{p})_y\otimes \OO_y/m_y=0$$
when $j<i/2$ 
by proposition \ref{prop:toricvan}.
\end{proof}

\section{Invariants of singularities}

The weight filtration can be  applied to the study of  the
topology of the fibres  of a resolution of singularities  and their links.

\begin{prop}\label{6.1} Let $Y$ be (the germ of) a variety with an  isolated singularity $y\in Y$.
Let $f: X\to Y$ be its resolution.
We can assign to $y$ the following invariants:
$$w^{i}_j(y):=\dim W_jH^i(f^{-1}(y))$$
These are independent of the choice of resolution $X$, where $j<i$.
\end{prop}

\begin{proof} 
We can assume that $Y$ is projective and that $U:=Y\setminus \{y\}$ is
nonsingular by
 resolving all points away from $y\in Y$, if necessary. This will not affect the exceptional fibre $f^{-1}(y)$ of the resolution.
Since $X$ is nonsingular we get $W_j(H^i(X))=W_j(H^{i+1}(X))=0$ for $j<i$. It follows from the long exact sequence 
$$\ldots W_jH_c^i(U)\to W_jH^{i}(X)\to W_jH^{i}(f^{-1}(y))\to W_jH_c^{i+1}(U)\to W_jH^{i+i}(X)\to \ldots$$
that $W_j(H^i(f^{-1}(y))\simeq W_j(H_c^{i+1}(U))$ is an
isomorphism. The latter is independent of the fibre of the resolution
$f: X\to Y$. 

The proposition can also be deduced from  \cite[7.1]{payne}.
\end{proof}

\begin{prop}\label{6} Let $Y$ be a singular variety  and
 $f: X\to Y$ be a resolution.
We can assign to  $y\in Y$ the following invariants, 
$$w_{0}^i(y):=\dim W_0H^i(f^{-1}(y))$$
$$h^1(y):=\dim H^1(f^{-1}(y))$$
independently of the choice of resolution $X$. 
\end{prop}
 
 \begin{rmk} These invariants generalize the cohomology of the dual graph of the isolated singularities. If $X\to Y$ is a resolution of an isolated singularity $y\in Y$, such that $D=f^{-1}(y)$ is a SNC divisor then the homotopy type of $\Sigma_D$ is independent of the resolution and, 
 in particular, $w_{0}^i(y)=h^i(\Sigma_D)$ are also independent of the resolution in this situation. Thus propositions \ref{6.1}, and \ref{6} generalize the observation for dual complexes to a more general situation. 
\end{rmk}
 
We first prove the following lemmas:

\begin{lemma} \label{prod} 
\begin{enumerate}
\-
\item If $X$ is projective nonsingular and irreducible, $Y$ is an
  arbitrary  variety  and $\pi: X \times Y \to X$ is the projection then the natural homomorphism
$$\pi^*:W_0(H^i(Y))\to W_0(H^i(Y\times X))$$ 
is an isomorphism.
\item If $X$ is projective nonsingular and irreducible with Betti
  number $b_1(X)=0$, $Y$ is an arbitrary  variety  and $\pi: X \times
  Y \to X$ is the  projection then the natural homomorphism
$$\pi^*:H^1(Y)\to H^1(Y\times X)$$ 
is an isomorphism.
\end{enumerate}
\end{lemma}
\begin{proof} (1) Note that, since $X$ is projective and nonsingular  we get $W_0(H^k(X))=0$, for $k>0$. Since $X$ is irreducible $W_0(H^0(X))\simeq \C$.
We have

 $W_0(H^i(Y\times X))=\bigoplus_{j+k=i} W_0(H^j(Y)\otimes W_0(H^k(X))=W_0(H^i(Y))\otimes W_0(H^0(X))=$
 
\noindent $W_0(H^i(Y))\otimes \C \simeq W_0(H^i(Y))$

(2) $H^1(Y\times X))= (H^0(Y)\otimes H^1(X))\oplus (H^1(Y)\otimes H^0(X))=H^1(Y)\otimes H^0(X))=H^1(Y)\otimes \C=H^1(Y)$
\end{proof}

\begin{lemma}\label{lemma:bundle} If $E\to Y$ is a Zariski locally trivial bundle with
  fibre $X$ then
\begin{enumerate}
\item if $X$ is nonsingular and projective,
$$\pi^*:W_0(H^i(Y))\to W_0(H^i(E))$$\noindent is an isomorphism.
\item if $b_1(X)=0$,
$$\pi^*:H^1(Y)\to H^1(E)$$\noindent is an isomorphism.
\end{enumerate}
\end{lemma} 
\begin{proof}
(1) We use induction on dimension of $Y$. Let $U\subset Y$ be an open set where the bundle $E_{|U}\simeq U\times X$ is trivial.  By the previous Lemma $$W_0(H^i(U))\to W_0(H^i(E_{|U})=W_0(H^i(U\times X))$$ \noindent is an isomorphism. Also, by the inductive assumption on dimension
$$W_0(H^i(Y\setminus U))\to W_0(H^i(E_{|Y\setminus U}))$$ \noindent is an isomorphism.
Apply the 5 lemma to the diagram  $$
\xymatrix{
\ldots W_0H_c^{i-1}(Y\setminus U)\ar[r]\ar[d]^{\simeq} &W_0H_c^i(U)\ar[r]\ar[d]^{\simeq} & W_0H^i_c(Y)\ar[r]\ar[d] & W_0H_c^i(Y\setminus U)\ar[d]^{\simeq}\ldots& \\ 
 W_0H_c^{i-1}(E_{|Y\setminus U})\ar[r] & W_0H_c^i(E_{|U})\ar[r] & W_0H_c^i(E)\ar[r] & W_0H_c^i(E_{|Y\setminus U})
}
$$
(2) We argue as in (1). We may further assume that $U$  contains no  compact connected components. Thus $H^0_c(U)=0$, and we  get 
$H^2(E_{|U})=H^2(U\times X)=(H_c^2(U)\otimes H_c^0(X))\oplus (H_c^1(U)\otimes H_c^1(X)) \oplus  (H_c^0(U)\otimes H_c^2(X))=H_c^2(U)\otimes H_c^0(X)=
H_c^2(U)\otimes \C\simeq H_c^2(U)$

We use the diagram. $$
\xymatrix{
\ldots H_c^0(Y\setminus U)\ar[r]\ar[d]^{\simeq} & H_c^1(U)\ar[r]\ar[d]^{\simeq} & H^1_c(Y)\ar[r]\ar[d] & H_c^1(Y\setminus U)\ar[r]\ar[d]^{\simeq}& H_c^2(U)\ar[d]^{\simeq} \ldots \\ 
H_c^{0}(E_{|Y\setminus U})\ar[r] & H_c^1(E_{|U})\ar[r] & H_c^1(E)\ar[r] & H_c^1(E_{|Y\setminus U})\ar[r] & H_c^2(E_{|U})
}
$$

\end{proof}
\begin{lemma}  If $f:X\to Y$ is a blow-up of a smooth centre on the nonsingular but not necessarily complete variety $Y$ then $f^*:W_0(H^i(Y))\to W_0(H^i(X))$ is an isomorphism.

\end{lemma}

\begin{proof} Let $C\subset Y$ be the smooth centre of the blow-up, and $E$ be the exceptional 
divisor. Then $E\to C$ is a locally trivial bundle with the fibre
isomorphic to ${\PP}^k$, for some $k$. Set $U=Y\setminus C= X\setminus E$.  Consider
the following diagram and use the 5 lemma.

$$
\xymatrix{
\ldots W_0H_c^{i-1}(C)\ar[r]\ar[d]^{\simeq} & W_0H_c^i(U)\ar[r]\ar[d]^{=} & W_0H^i_c(Y)\ar[r]\ar[d] & W_0H_c^i(C)\ar[d]^{\simeq}\ldots \\ 
 W_0H_c^{i-1}(E)\ar[r] & W_0H_c^i(U)\ar[r] & W_0H^i_c(X)\ar[r] & W_0H_c^i(E)
}
$$

\end{proof}

\begin{proof}[Proof of Proposition \ref{6}]
By the Weak Factorization theorem (\cite{wlodarczyk}, \cite{akmw}), any two desingularizations can be connected
by a sequence of blow-ups with smooth centres. Therefore it suffices
to compare  two resolutions $f:X\to Y$, and $\tilde{f}: \tilde{X}\to
Y$, such that $\tilde{X}\to X$ is the blow up along a smooth centre $C$.
Let $U= X\setminus f^{-1}(y)$, $\tilde{U}=\tilde{X}\setminus
\tilde{f}^{-1}(y)$ and let $E$ be the exceptional divisor of $\tilde{X}\to X$.

(1) Consider the diagram
 
 $$
\xymatrix{
\ldots W_0H^{i-1}(f^{-1}(y))\ar[r]\ar[d] & W_0H_c^i(U)\ar[r]\ar[d]^{\simeq} & W_0H^i(X)\ar[r]\ar[d]^{\simeq} & W_0H^i(f^{-1}(y))\ar[d]\ldots \\ 
 W_0H^{i-1}(\tilde{f}^{-1}(y))\ar[r] & W_0H_c^i(\tilde{U})\ar[r] & W_0H_c^i(\tilde{X})\ar[r] & W_0H^i(\tilde{f}^{-1}(y))
}
$$
It folows from lemma \ref{lemma:bundle} that  $W_0H_c^i(U)\to W_0H_c^i(\tilde{U})$ and  $W_0H_c^i(X)\to W_0H_c^i(\tilde{X})$ are isomorphisms.  By the diagram and 5 lemma we get that $$W_0H^i(f^{-1}(y))\to W_0H^i(\tilde{f}^{-1}(y))$$  is an isomorphism.

(2) Consider the diagram
 
 $$
\xymatrix{
\ldots H^{1}(f^{-1}(y)\setminus C)\ar[r]\ar[d]^{=} &H^{1}(f^{-1}(y))\ar[r]\ar[d] & H_c^1(f^{-1}(y)\cap C)\ar[r]\ar[d]^{\simeq}  & H^1(f^{-1}(y)\setminus C))\ar[d]^{=}\ldots \\ 
 H^{0}(\tilde{f}^{-1}(y)\setminus E)\ar[r] & H^{0}(\tilde{f}^{-1}(y))\ar[r]& H_c^1(\tilde{f}^{-1}(y)\cap E)\ar[r]  & H^1(\tilde{f}^{-1}(y)\setminus E)
}
$$
Note that  $\tilde{f}^{-1}(y)\cap E\to f^{-1}(y)\cap C$ is a locally
trivial $\PP^k$-bundle. By the diagram and 5-Lemma we get that $$H^1(f^{-1}(y))\to H^1(\tilde{f}^{-1}(y))$$  is an isomorphism.
\end{proof}

In general, if $j>0$, $\dim W_jH^i(f^{-1}(y))$ may depend upon the resolution $f$.
That is why we extend the above definition in two ways:
$$w_{j}^i(y):=\inf_{f: X\to Y} \dim W_jH^i(f^{-1}(y))$$
$$\bar{w}_{j}^i(y):=\sup_{f: X\to Y} \dim W_jH^i(f^{-1}(y))$$
$$h^{i}(y):=\inf_{f: X\to Y} \dim H^i(f^{-1}(y))$$
$$\bar{h}^{i}(y):=\sup_{f: X\to Y} (\dim H^i(f^{-1}(y))$$
where $f:X\to Y$ varies over all resolutions above.
 It follows immediately from the definition and the previous theorems
 that we have the following properties of the invariants $w_{j}^i(y)$,
 and $\bar{w}_{j}^i(y)$, $h^i(y)$, $\bar{h}^i(y)$.

 \begin{prop} \label{main2}
\-
\begin{enumerate}
 \item Let
 $f: X\to Y$ be a  resolution of $Y$.
Then $$ 0 \leq w_{j}^i(y)\leq \dim W_jH^i(f^{-1}(y))\leq \bar{w}_{j}^i(y)\leq \infty$$
$$ 0 \leq h^{i}(y)\leq \dim H^i(f^{-1}(y))\leq \bar{h}^{i}(y)\leq \infty$$
\item $$0\leq w_{0}^i(y)\leq w_{1}^i(y)\leq \ldots \leq w_{i}^i(y)=h^i(y)=w_{i+1}^i(y)=\ldots,$$ and  
$$0\leq \bar{w}_{0}^i(y)\leq \bar{w}_{1}^i(y)\leq \ldots \leq \bar{w}_{i}^i(y)=\bar{h}^i(y)=\bar{w}_{i+1}^i(y)=\ldots.$$

\item If $g:Y_1\to Y_2$ is a smooth morphism
and $y\in Y_1$ is  a point such that 
$w_{j}^i(y)=\bar{w}_{j}^i(y)$, (respectively $h^i(y)=\bar{h}^i(y)$) then  $$w_{j}^i(g(y))=\bar{w}_{j}^i(g(y))=w_{j}^i(y)=\bar{w}_{j}^i(y)$$ (respectively $$h^i(g(y))=\bar{h}^i(g(y))=h^i(y)=\bar{h}^i(y)).$$

\item 
The equality $w_{j}^i(y)=\bar{w}_{j}^i(y)$ holds if 
\begin{enumerate} \item $y$ is an isolated singularity and $j<i$, or $j=0$,\item  $y$ is arbitrary, and $j=0$, or $i\leq 1$.
\end{enumerate}
\item $w_{j}^i(y)=\bar{w}_{j}^i(y)=0$\quad  if \quad  $i-j  \geq \dim(Y)-1$
\item $h^i(y)=\bar{h}^i(y)$ if $i\leq 1$.
\item $\bar{h}^{i}(y)=\infty$ for   $2\leq i \leq  2\dim_yY-2$
\item $h^i(y)=\bar{h}^i(y)=0$ for $i\geq 2\dim_yY-1$

\end{enumerate}
\end{prop}
\begin{proof} (1) is obvious. (2) and (5) follows from corollary~\ref{cor:weight}.

(3) Let  $g:Y_1\to Y_2$ be a smooth morphism. If $X_2\to Y_2$ is a resolution then
$X_1:=Y_{1\times Y_2} X_2 \to Y_1$ is also a resolution of $Y_1$ with the same fibers,
and there is a fiber square :
$$
 \xymatrix{
  X_1\ar[d]\ar[r] & X_2\ar[d] \\ 
  Y_1\ar[r] & Y_2
 }
 $$

 of resolutions with horizontal smooth maps and
 the result follows.
 
 (4) and (6) follow from Propositions \ref{6.1} and\ref{6}.

  (7) Blow-up  a smooth centre $C$ in the fibre $f^{-1}(y)$ of the
  resolution $f:X\to Y$ to obtain a new resolution $\tilde{f}:
  \tilde{X}\to Y$. Let $E$ denote the exceptional fibre of $\tilde{f}$.
Consider Mayer-Vietoris sequence 
$$
\ldots \to W_jH_c^{i-1}(E)\to W_jH_c^i(X)\to
W_jH_c^{i}(\tilde{X})\oplus W_jH_c^{i}(C)\to W_jH_c^{i}(E) \ldots
$$
If $j=i$  then $ W_jH_c^{i}()=H_c^{i}()$, and $W_{j}H^{i+1}(X)=0$. Thus 
$H_c^{i}(\tilde{X})\oplus H_c^{i}(C)\to H_c^{i}(E)$ is an epimorphism.
Since $\tilde{f}^{-1}(y)\supseteq E$, the above morphism factors through
$$H_c^{i}(\tilde{X})\oplus H_c^{i}(C)\to H_c^{i}(\tilde{f}^{-1}(y))\oplus H_c^{i}(C)\to H_c^{i}(E).$$
Consequently the Mayer-Vietoris
morphism $$H_c^{i}(\tilde{f}^{-1}(y))\oplus H_c^{i}(C)\to H_c^{i}(E)$$ 
 is also an epimorphism for all $i$.  Thus, in the Mayer-Vietoris sequence.
$$
\ldots \to H_c^{i-1}(E)\to H_c^i(f^{-1}(y))\buildrel \psi\over\to
H_c^{i}(\tilde{f}^{-1}(y))\oplus H_c^{i}(C)\buildrel \phi\over\to H_c^{i}(E) \ldots
$$
we see $\phi$ is an epimorphism and $\psi$ is a
monomorphism. Therefore we get the short exact sequence.
$$
0 \to H_c^i(f^{-1}(y))\buildrel \psi\over\to
H_c^{i}(\tilde{f}^{-1}(y))\oplus H_c^{i}(C)\buildrel \phi\over\to H_c^{i}(E) \to 0$$
We are going to use blow-ups with  centers which are either point or smooth elliptic curves.
Then $\dim(C)\leq 1$, and $H_c^{i}(C)=0$ for $i> 2$.
We get $$\dim(H_c^{i}(\tilde{f}^{-1}(y))=\dim(H_c^{i}({f}^{-1}(y))+\dim(H_c^{i}(E)).$$
Since $E$ is a locally trivial bundle over $C$ with fibre ${\PP}^l$,
where $n=\dim(Y)$, $l=n-1$ if $C$ is a point, and $l=n-2$ if $C$ is an
eliptic curve. As is well known, e.g. \cite[p 270]{bott},
$$H_c^{i}(E)\simeq \bigoplus_{j+k=i} H^j(C)\otimes H^k({\PP}^l).$$

In particular the Poincare polynomial $P_E$ of $E$ is equal to
$$P_E(t)=1+t^2+\ldots+t^{2n-2}$$\noindent for the blow-up at the point, and
$$P_E(t)=(1+2t+t^2)(1+t^2+\ldots+t^{2n-4})$$\noindent for the blow-up at the elliptic curve.
 If we apply blow-ups at points and elliptic curves in  in the fibre $f^{-1}(y)$ we can increase cohomology $h^i(f^{-1}(y))$, where $2\leq i\leq 2n-2$.

(8) Follows from the inequality $\dim_{\R}(f^{-1}(y))\leq 2\dim(Y)-2$.
\end{proof}
\begin{thm}
\-
\begin{enumerate}
\item If $y\in Y$ is a nonsingular point
  then $$h^i(y)=w_{j}^i(y)=0,\mbox{ for}\quad i>0.$$
 $$\bar{w}_{j}^i(y)=w_{j}^i(y)=0\quad \mbox{for}\quad i>0,\quad  j<i.$$  
 \item If $y\in Y$ is toroidal   (analytically equivalent to the germ of a toric
variety) then $$w_{j}^i(y)=\bar{w}_{j}^i(y)=0$$ \noindent for  $i>0$, $j<i/2$.
\item If $y\in Y$ is rational then    $$w_{0}^i(y)=\bar{w}_{0}^i(y)=0$$  for $i>0$. 
\item If $y \in Y$ is an isolated normal Cohen-Macaulay singularity  $$w_{0}^i(y)=\bar{w}_{0}^i(y)=0 $$ \noindent for  $ i\neq \,0, \, \dim Y-1$ 
\item If $y \in Y$ is normal then $h^{0}(y)=1$.
\end{enumerate}

\end{thm} 
\begin{proof}
(1) To see this, take  the trivial resolution $Y\buildrel {\rm id}
\over\to Y$, and  apply (4) from the previous proposition.

(2) follows from Proposition  \ref{prop:toric}.

(3) and (4) follow from Corollary \ref{cor:main}. 

(5) follows from Zariski's main Theorem.
\end{proof}

\section{Weights of the link}

The link $L$ of a singularity $(Y,y)$ is the boundary of a suitable small contractible
neighbourhood of $y$.  When $(Y,y)$ has isolated singularities,
$H^i(L)$ carries  mixed Hodge structure  by the identification
$$H^i(L)\cong H^{i+1}(Y,Y-\{y\})$$ when $i>0$ (cf. \cite{steenbrink}).
More generally, work of Durfee and Saito \cite{ds} shows that
the  intersection cohomology $IH^i(L)$ carries a mixed Hodge structure
which is independent of any choices. (Among the various indexing
conventions, we choose the one where  $IH^i(L)$ coincides with ordinary
cohomology $H^i(L)$ when  $L$ is a manifold, i.e. when
$(Y,y)$ is isolated. This means that $IH^i(L) =
H^{i-\dim L}(k_{!*}(\Q_{L'}[\dim L]))$, where $k:L'\to L$ is the smooth locus.)
The general case, which appeals to Saito's  theory of mixed Hodge
modules \cite{saito}, is much more involved. By definition the mixed
Hodge structure on the intersection cohomology of the link is given
by right hand side of equation \eqref{eq:link}, below,  computed  in the category of mixed Hodge modules.
We can apply our previous results to get bounds on the weights of these mixed Hodge structures.

These lead us to the following  interesting invariants of singularities:

$$\ell_{j}^i(y):=\dim W_jIH^i(L),$$
\noindent where  $L$ is the link of singularity $(Y,y)$.

\begin{thm}
If $i<\dim Y$, then
$$\ell_{j}^i(y)\le \dim W_j(H^i(f^{-1}(y))$$
for any desingularization $f:X\to Y$.
In particular, for any singularity $y\in Y$, $$\ell_{j}^i(y)\le w_{j}^i(y).$$
\end{thm}

\begin{proof}
Let $n=\dim Y$.
  By the decomposition theorem \cite{bbd}, $\R f_* \Q[n]$ decomposes
  into a sum of shifted perverse sheaves. This moreover holds in the
  derived category of mixed Hodge modules by the work of Saito
  \cite{saito}. By restricting to the smooth locus $k:Y'\to Y$, we
  can see that one of these summands is necessarily the intersection
  cohomology complex $IC_Y= k_{!*}\Q_{Y'}[n]$. Let $U=Y-y$ and denote
  the  inclusions by  $j:U\to Y$ and $\iota:y\to Y$.
Then from \cite[2.1.11]{bbd}, we can conclude
  that $IC_Y= \tau_{\le -1} \R j_*IC_U$, where $\tau_\dt$ is the
  standard truncation operator \cite[1.4.6]{deligne}.  Since
  \begin{eqnarray}
    \label{eq:link}
    IH^{i+n}(L) &=& H^i( \iota^*\R j_*IC_U) \\
&=&H^i(\iota^*(\tau_{\le -1} \R
 j_*IC_U))
  \end{eqnarray}
for $i<0$, it follows
  that $IH^{i+ n}(L)$ is a summand of $H^i(\R f_* \Q[n] )=(H^{i+n}(f^{-1}(y))$ when $i<0$.
\end{proof}

As a corollary we get 
\begin{thm}
If $i<\dim Y$, then
$$\ell_{j}^i(y) \le w_{j}^i(y)\le \sum_{p\le j} 
\dim  (R^{i-p}f_*\Omega_X^{p})_y\otimes \OO_y/m_y$$
for any desingularization $f:X\to Y$.

\end{thm}
\begin{proof} We use the previous theorem  and  apply the
  bounds from corollary \ref{cor:2ndthm}.
\end{proof}

\begin{rmk}
  For an isolated singularity, we can argue more directly, without
  mixed Hodge module theory, as in \cite[cor 1.12]{steenbrink}.
\end{rmk}
\begin{ex} If $x\in X$ is a nonsingular point then $L\simeq S^{2n-1}$ and we get

$\ell_{j}^i(x)=0$ for $j<i$, or  $j\geq i$ and 
$i\neq 0,2n-1$,

 $\ell_{j}^i(x)=1$ if $j\geq i$ and $i=0,2n-1$
\end{ex}
\begin{cor}
For a normal isolated Cohen-Macaulay (respectively
rational) singularity, 
  $$w_{0}^i(y)=\bar{w}_{0}^i(y)=\ell_{0}^i(y)=0 $$ \noindent for $0<i<\dim Y-1$ (respectively $i>0$).
If $(Y,y)$ is toroidal then $$w_{j}^i(y)=\bar{w}_{j}^i(y)= \ell_{j}^i(y)=0$$ \noindent for $j<i/2$.
\end{cor}

The natural conjecture which arises here is whether 
\begin{conj} The invariants $w_{j}^i(y), \bar{w}_{j}^i(y),
  \ell_{j}^i(y) $ are upper semicontinuous.
\end{conj}

\bigskip
\section{Boundary divisors and the Weak factorization theorem}

Suppose we are given  a smooth complete variety $X$, and  a  divisor with simple normal crossings
$E\subset X$, or more generally
a union of smooth subvarieties  which is local analytically a union of
intersections of coordinate hyperplanes.
We will show that the homotopy type of the
dual complex of $E$  depends only the  the
complement $X-E$, and in fact only on its proper birational class.
In fact, we will prove somewhat sharper  results in theorems  \ref{8}, \ref{9} and \ref{10}  below.
Stepanov in \cite{step2} showed this result in the particular case,
where $ E$ is the exceptional divisor of a resolution of an isolated
singularity. Actually, his proof works for any boundary divisors which
are complements a fixed open subset $U=X- E$.  Thuillier in
\cite{th} proves  a similar  result in any characteristic. In our
set up,  the complement $X- E$ is not fixed. That is, we
assume that for two different divisors $E$ and  $E'$, there is  a
proper birational map $X- E\dashrightarrow X'-E'$. This implies the isomorphism of homotopy types of the dual
complexes of $E$ and $E'$. Payne,  in his paper  \cite[\S 5]{payne},
 compares the homotopy types of the dual complexes of divisors of  different log resolutions of a given singular variety.
In our version, the situation is more general.  In particular, we allow the
maximal components of $E$ to be of any dimension. This situation
arises naturally when $E$ is the fibre of a resolution of a  nonisolated singularity.

We recall some constructions and notations from earlier sections, so
that section can be read independently of the rest of the paper, Let
$D=\cup D_i$ be a SNC (simple normal crossing) divisor on a
nonsingular variety $X$. The dual complex is a CW-complex $\Sigma_D$ whose cells are simplices $\Delta^j_{i_1,\ldots,i_k}$ corresponding to the irreducible components of $D^j_{i_1,\ldots,i_k}$ of the intersection $D_{i_1,\ldots,i_k}=D_{i_1}\cap,\ldots\cap D_{i_k}$. 
 For any $\Delta=\Delta^j_{i_1,\ldots,i_k}$, and $\Delta'=\Delta^{j'}_{i'_1,\ldots,i'_s}$, where $\{{i'_1,\ldots,i'_s}\}\subset \{i_1,\ldots,i_k\}$ the simplex $\Delta'$ is a face of $\Delta$ if  $D^{j'}_{{i'_1,\ldots,i'_s}}\supset D^j_{i_1,\ldots,i_k}$. 
If all the intersections $D_{i_1,\ldots,i_k}=D_{i_1}\cap,\ldots\cap D_{i_k}$ are irreducible then then $\Sigma$  is a simplicial complex. 
In general the dual complex is a {\it quasicomplex}, i.e. it is a collection of simplices closed with respect to the face relation, and such that the intersection of two simplices is a union of some of their faces.

We recall some basic notions for quasicomplexes which refine those for
simplicial complexes.

\begin{itemize}
\item By the $\Star(\Delta,\Sigma)$, where  $\Delta\in \Sigma$, we mean the set of  all faces of $S$ which contain $\Delta$.

\item For any set $S\subset\Sigma$ by $\overline{S}\subset\Sigma$ we mean the complex consisting of simplices of $\Sigma$ and their faces.

\item By the {\it link}  we mean $L(\Delta,\Sigma_D)= \overline{\Star(\Delta,\Sigma_D)}\setminus  \Star(\Delta,\Sigma_D)$.

\item Let $v_\Delta$ be the barycentric centre of $\Delta$.
By the {\it stellar sudivision} of $\Sigma$ at $\Delta$ (or at $v_\Delta$) we mean
 $$\Sigma':=v_\Delta\cdot \Sigma:=\Sigma \setminus
 \Star(\Delta,\Sigma)\cup \overline{\{\conv(v,\Delta)\mid \Delta \in
   L(\Delta,\Sigma)\}},$$
where  $\conv(v_\Delta,\Delta)$  is the  simplex spanned by $v_\Delta$, and  $\Delta$.
\end{itemize}

Let   $\Sigma_D$  be a dual (quasi)-complex associated with $D$ on a
nonsingular $X$. Let  $\Delta$ be a simplex in $\Sigma_D$ and let
$C=D(\Delta)$ be the corresponding intersection components of $D$.  Then  the blow-up $\sigma: X'\to X$ of $C$  determines transformation of divisors $D\mapsto D'=\sigma^{-1}(D)$ which corresponds to the stellar subdivision $v_\Delta\cdot\Sigma_D$. 
Recall a well known lemma:

\begin{lemma} Let   $\Sigma_D$  be the dual (quasi)-complex associated with $D$ on a nonsingular $X$. Let $X'\to X$ be a composition of blow-ups of all the intersection components starting from the components of the smallest dimension
and ending at the blow-ups of the components of the highest dimension (divisors). 
Then the resulting quasicomplex $\Sigma_{D'}$ which is obtained from $\Sigma_D$ by the successive stellar  subdivisions is a simplicial complex.
\end{lemma}
\begin{proof}  In the process we blow up (and thus eliminate) all the strict transforms of the intersections components of $D=\bigcup D_i$. Note that  the center $C= D^j_{i_1,\ldots,i_k}$ is the lowest dimensional intersection component of the strict transforms of $D_i$  so it intersects no other divisors $D_i$, except for $i={i_1,\ldots,i_k}$.

By  the induction centers of blow-ups and thus new exceptional divisors have irreducible intersections with the strict transforms of the intersections components of D and the already created exceptional divisors E. 
 Finally we   will have only irreducible intersections of the exceptional divisors.

\end{proof}

By this  lemma, it always possible to reduce the situation to SNC
divisor whose associated dual quasicomplex is a complex by applying
additional blow-ups. We will not need to do this however.

\begin{rmk} The blow-ups of the divisorial components coresponding to
  the stellar subdivisons at the vertices define identity
 transformations. They are introduced to simplify the considerations.
\end{rmk}

Proposition \ref{6} can be generalized in a few ways.
\begin{prop}\label{65} If $\phi :X\dashrightarrow  X'$ is a proper  birational map of two nonsingular projective varieties, such that its restriction $\phi_{|U}: U\dashrightarrow U'$ to   open sets $U\subset X$ and $U'\subset X'$ is proper. Then there is a  natural isomorphism
$$W_0(H^i(X\setminus U))\simeq W_0(H^i(X'\setminus U'))$$
\end{prop}
\begin{proof} This is identical to the proof of Proposition \ref{6}.
\end{proof}

If the fibre $f^{-1}(y)=D_y$ of the resolution $f:X\to Y$ is a SNC
divisor then the above Proposition \ref{65} says that the cohomology 
$H^i(\Sigma_{D_y})=W_0(H^i(D_y))$ of the dual complex $\Sigma_{D_y}$ are independent of the resolution  $f:X\to Y$ .
This  suggests:

\begin{thm}\label{7}  Let $f^{-1}(y)=D$ be the fibre of the resolution $f:Y\to X$ which is a SNC divisor.
 The homotopy type of the dual complex  $\Sigma_D$ corresponding to  the   fibre
 is independent of the resolution $f:Y\to X$.
\end{thm}

When the singularity is nonisolated, the fibres are usually not
divisors. We consider a generalization later in theorem \ref{10},
which allows for general fibres.
The theorem above is a generalization of  the version of theorem for
isolated singularities given in \cite{step, step2}. 
The following theorems  \ref{8}, \ref{9} generalize  theorem \ref{7}, and  
the results  of \cite{payne} 
for log-resolutions.

\begin{thm} \label{8} If $\phi :X\dashrightarrow  X'$ is a birational map of two nonsingular projective varieties, such that its restriction $\phi_{|U}: U\dashrightarrow  U'$ to   open subsets $U\subset X$ and $U'\subset X'$ is proper. 
Assume that the boundary divisors  $D=X\setminus U$ and $D'=X'\setminus U'$ have SNC.
Then the dual complexes $\Sigma_D$ and $\Sigma_{D'}$ are homotopically equivalent
\end{thm}

The method of the proof is identical with the one used in Stepanov  \cite{step2}.
We just need  the following  simple strengthening of the Weak
Factorization Theorem (\cite{wlodarczyk,akmw}).

\begin{thm}\label{fact} (\cite{wlodarczyk,akmw}) Let $\phi:
  X\dashrightarrow X'$ be a birational map of smooth complete
  varieties such that its restriction $\phi_{|U}: U\to U'$ to   open
  susbsets $U\subset X$ and $U'\subset X'$ is proper and which is an
  isomorphism over an open subset $V\subset U,U'$. 
Assume that the boundary divisors  $D=X\setminus U$ and $D'=X'\setminus U'$ have SNC.
Then there exists
a weak factorization,  that is a sequence of birational maps $$X=X_0\buildrel \phi_0 \over \dashrightarrow  X_1
\buildrel \phi_1 \over \dashrightarrow \ldots \buildrel \phi_{n-1} \over
\dashrightarrow X_n=X'  ,$$ 
 such that  $\phi_i: X=X_i \dashrightarrow  X_{i+1}$ is either a blow-up or a blow-down  along smooth centres. Moreover
 all the complements $D_i:= X_i\setminus U_i$ are SNC divisors and all
 the centres have SNC with respect to $D_i$. Finally, the centres are disjoint from $V$.
\end{thm}

We prove a more general version:

\begin{thm}\label{fact2} (\cite{wlodarczyk, akmw}) Let $\phi : X\dashrightarrow  X'$ be a birational map of smooth complete varieties. Let $D^1\subset D^2\subset\ldots \subset D^k=D$ be SNC divisors on $X$, and 
$D'^1\subset D'^2\subset\ldots \subset D'^k=D'$ be SNC divisors on $X'$
 such that for $U^i:=X\setminus D^i$, and $U'^i:=X'\setminus D'^i$  the  restriction $f_{|U^i}: U^i\to U'^i$  is proper. Assume $\phi$ is an isomorphism over an open subset $V\subset U^k, U'^k$. Then there exists
a weak factorization 
 $$X=X_0\buildrel \phi_0 \over \dashrightarrow  X_1
\buildrel \phi_1 \over \dashrightarrow \ldots \buildrel \phi_{n-1} \over
\dashrightarrow X_n=X'  ,$$
as above, 
 such that for the open subsets $U^j_i\subset X_i$ all the complements of $D^j_i:= X_i\setminus U^j_i$ are SNC divisors and all the centres have SNC with $D^j_i$.
\end{thm}

\begin{proof}  By a version of   Hironaka principalization for an
  ideal sheaf $\cI_{X\setminus V}$  (respectively $\cI_{X'\setminus
    V}$)  and SNC divisor $D$ (respectively  $D'$ )   there exists a  sequence
  of blow-ups  $\pi: Y\to X$ (respectively $\pi': Y'\to X'$) with 
smooth centres having SNC  with the inverse images  of $D$
(respectively $D'$). Moreover, the total transform of  $\cI_{X\setminus V}$
(respectively  of $\cI_{X'\setminus V}$ ) 
is an ideal of an SNC divisor whose components have SNC with the inverse
images of $D$ (resp. $D'$)  (cf. \cite{kollar},\cite{wlodarczyk3}). 

There exists a weak factorization of $f:Y\to Y'$ such that all
intermediate steps $Y_i$ admit  a morphism  $f_i: Y_i\to Y$ or  $f'_i:
Y_i\to Y'$ (\cite{wlodarczyk}, \cite{akmw}). Thus  the full  transform
of $f_i^{-1}\pi^{-1}(\cI_{D^j})$ (or
$f'^{-1}_i\pi'^{-1}(\cI_{D'{^j}})$) is principal and the complement $D^j_i:=f_i^{-1}\pi^{-1}(D^j)=Y_i\setminus U^j_i$ is  of codimension one for any $j=1\ldots k$. Moreover all centres of the blow-ups have SNC with the complement divisor $E_i:= Y_i\setminus  V$ containing $D^j_i$. Thus $D^j_i$ is a SNC divisor on $Y_i$ and all the centres in the Weak Factorization have SNC with $D^j_i$.
\end{proof}
\vskip .1in

\begin{proof}[Proof of Theorem \ref{8}] The proof is an extension of the method mentioned in \cite{step}.
 By theorem \ref{fact2},  we can connect $D$ and $D'$ by blow-ups with smooth centres
which have SNC with intermediate divisors which are complements of  $U$.

Thus it is sufficient  to consider the effect of a single blow-up.
This is already done  in \cite{step2}.
We describe  this  in order to keep the presentation self contained and to provide a model for
what comes later.
We also introduce some convenient notation here.
Consider local coordinates $x_1,\ldots,x_n$ on an open neighborhood $U_p\subset X$  of some point $p\in C\cap D$. We can assume  that the coordinates $x_1,\ldots,x_k$ describe the components of the divisor $D$, and that the centre $C$ is described by $x_r,\ldots,x_s$, where $r\leq s$.
Consider three cases.

{\bf Case 1.} The centre $C$ is not contained in $D$ (but has SNC with $D$).

 This means that $C$  is not contained in a component $D_i$. Thus   no $x_i$, where $i\leq k$ vanishes on $C$, and consequently   $k< r$.
In this case the blow-up is defined by a  coordinate transformation $$x_1,\ldots,x_k,\ldots, x_{s-1}, x_{s}/x_r,\ldots, x_{r-1}/x_r, x_r,x_{r+1}\ldots,x_n$$ \noindent  which do not change the configuration of components and thus the dual complex of the exceptional divisor.
$$\Sigma_{D'}=\Sigma_D.$$

{\bf Case 2. } The centre $C=D_r\cap\ldots\cap D_s$ is the intersection of some divisorial components and corresponds to the face $\Delta_C:=\Delta_{r,\ldots,s}$. ($s<k$. )

 The blow-up of $C$ determines the stellar subdivision $\Sigma_{D'}=v_C\cdot \Sigma_D$ of $\Sigma_D$ at the centre $v_C$ of $\Delta_C=\Delta_{r,\ldots,s}$. 
 Thus topologically $\Sigma_D$ and $\Sigma_{D'}$ are homeomorphic.

{\bf Case 3.}  The centre $C$ is contained in the intersection $ D_r\cap\ldots\cap D_k$ and is not contained in any smaller intersection. ( $r<k\leq s$)

 Extend locally the set of SNC divisorial components of $D$ to the the set all divisors $D\cup E$ corresponding to the complete coordinate system $x_1,\ldots, x_n$.
 
Consider the simplex $\Delta$ corresponding to the complete system of coordinates $x_1,\ldots,x_n$.
Its face $\Delta_{1,\ldots,k}$ is also a  face of  $\Sigma_D$. 
Let  $\Delta_{D,C}:=\Delta_{s,\ldots,k}\subset \Delta_{1,\ldots,k}$ be the face of $\Sigma_D$ corresponding to the minimal intersection component $ D_r\cap\ldots\cap D_k$ which contain $C$.
The blow-up of $C$ corresponds to the stellar subdivision at the centre $v_C:=v_{s,\ldots,r}$ of the face $\Delta_C:=\Delta_{s,\ldots,r} \in \Sigma_E$.
After the stellar subdivision the face  $\Delta_{1,\ldots,k}$ of  $\Sigma_D$ remains unchanged . We introduce the new vertex $v_C=v_{s,\ldots,r}$ in the dual complex $\Sigma_D$ corresponding to the  exceptional divisor $E$. 
Denote by $D(\Delta)$, where $\Delta\in \Sigma_D$, the  stratum which is the intersection of divisors corresponding to vertices of $\Delta$.

Set $$\Sigma_{D,C}:=\{\Delta\in \Sigma_D \mid D(\Delta)\cap C \neq \emptyset\} $$
Note that by definition  if  $D(\Delta)\cap C \neq \emptyset$ then $D(\Delta)\cap D(\Delta_C) \neq \emptyset$.
Thus  

(*) the maximal simplexes of $\Sigma_{D,C}$ are in $\Star(\Delta_{D,C},\Sigma_D)$.

The new dual complex is given  by
$$\Sigma_{D'}=\Sigma_D\cup \,\,v_C*\Sigma_{D,C},$$
where
$$v_C*\Sigma_C:=\overline{\{\conv(v_C,\Delta)\mid \Delta \in \Sigma_C\} }$$
is a cone over $\Sigma_C$ with vertex $v_C$. In addition
$\conv(v_C,\Delta)$ is the simplex spanned by $v_C$, and  $\Delta$ (a cone over $\Delta$ with vertex $v_C$).

Let $v_{D,C}$ denote the barycentre of $\Delta_{D,C}$.
 The  complex $\Sigma_{D'}$ is homotopy equivalent to $\Sigma_D$. In
 one direction, the map  $\beta:\Sigma_{D'}\to\Sigma_{D}$ is defined
 on the vertex $v_C\to v_{D,C}$ retracts $v_C* \Sigma_C$ to the
 $\Sigma_C$ and is identical on $\Sigma_D$, and defines homotopy
 equivalence. The homotopy inverse map $\alpha:\Sigma_{D}\to \Sigma_{D'} $ is given by the inclusion.
Then $\beta\alpha=id:  \Sigma_{D} \to \Sigma_{D} $  and   $\alpha\beta: \Sigma_{D'}\to \Sigma_{D'}$  is homotopic to the identity via
$$v_C\mapsto (1-t)v_{D,C}+tv_C.$$
Note that, by the condition (*), the interval $[v_{D,C}, v_C]$ is contained in $\Delta_{C}$ as well as in its subset $\conv(\Delta_{D,C}, v_C)\in\Sigma_{D'}$.
The point $v_{D,C}$ in the above construction can be replaced by any other point in $\Delta_{D,C}$. 
\end{proof}

The above theorem can be easily extended to  multiple divisors. This case was also considered in \cite{payne} in the context of log resolutions, and it  is nearly the same as the previous one.
Theorem \ref{9} generalizes and implies the results in \cite{payne}. 

\begin{thm}\label{9}  Let $\phi: X\dashrightarrow X'$ be a birational map of smooth complete varieties. Let $D^1\subset D^2\subset\ldots \subset D^k$ be SNC divisors on $X$, and 
$D'^1\subset D'^2\subset\ldots \subset D'^k$ be SNC divisors on $X'$
 such that for $U^i:=X\setminus D^i$, and $U'^i:=X'\setminus D'^i$  the  restriction $\phi_{|U^i}: U^i\dashrightarrow  U'^i$  is proper. 
Then the corresponding sequences of topological spaces $\Sigma^1\subset\Sigma^2\subset\ldots \subset \Sigma^k$ and  $\Sigma'^1\subset\Sigma'^2\subset\ldots \subset \Sigma'^k$ 
are homotopically equivalent.

\end{thm}
\begin{proof}
As before, by the Weak factorization theorem it suffices to consider the effect of a single blow-up .
We use the notation from the previous proof.

Assume $C=D(\Delta_C)$ is the intersection of some components of $D^k$. In the other cases the reasoning is the same.
Let $r:=\min\{i \mid C\subset D^i\}$ and 
$ \ell:=\min \{i\mid  C=D(\Delta_{C,\Sigma^i})\}$ be the smallest index $j$ such that the centre $C$ is the intersection of the divisors in $D^j$.

 Define for $i=r,\ldots,k$ the sequence of
 subcomplexes $$\Sigma_C^i:=\{\Delta\in \Sigma^i \mid D(\Delta)\cap C
 \neq \emptyset\} .$$
Then  $$\emptyset\subset \Sigma^r_C\subset\ldots \subset \Sigma^\ell_C=\overline{\Star(\Delta_C,\Sigma_\ell)}\subset \ldots \subset \Sigma^k_C=\overline{\Star(\Delta_C,\Sigma_k)}$$

The blow-up transforms $\Sigma^1\subset\Sigma^2\subset\ldots \subset \Sigma^k$
into $\Sigma'^1\subset\Sigma'^2\subset\ldots \subset \Sigma'^k$, where 
\begin{enumerate}
\item $\Sigma'^i=\Sigma^i$ for $1 \leq i\leq r$
\item $\Sigma'^i=\Sigma^i\cup (v_C* \Sigma^i_C)$ for  $r\leq i< \ell$ 
\item $\Sigma'^i=v_C\cdot \Sigma^i$ for $\ell \leq i\leq k$  .
\end{enumerate}
Let $\Delta_{C,r}\in \Sigma_r$ be the face corresponding  to the smallest intersection component of $D^r$ containing $C$.  Denote by $v_{C,r }$ the barycentre of  $\Delta_{C,r}\in \Sigma_r$.
There exists a map of topological spaces $\alpha: \Sigma^k\to \Sigma'^k$ which is an identity, and whose restrictions
$\alpha_i:  \Sigma^i\to \Sigma'^i$ are inclusions for $r\leq i<\ell$, and identities for $1\leq i<r$, and $\ell\leq i\leq k$.  There exists a map $\beta: \Sigma'^k\to \Sigma^k$ which is defined
on the vertex 
$$v_C\mapsto v_{C,r}\in \Sigma^r\subset \Sigma^{r+1}\subset\ldots\subset \Sigma^k$$ 
and  identity on other vertices.
Its restriction $\beta_i: \Sigma'^i\to \Sigma^i$ is a linear map homotopic to the identity for $\ell\leq i\leq k$, a
 retracts $v_C* \Sigma^i_C$ to the $\Sigma^i_C$ and  $\Sigma'^i$ to  $\Sigma^i$ for $r\leq i<\ell$ and is identity for $i\leq r$.

The homotopy of  the compositions  $\alpha\beta: \Sigma^k\to\Sigma^k$ (and $\beta\alpha:\Sigma'^k\to\Sigma'^k$) with the identities are defined by 
 $$v_C\to (1-t)v_{C,r}+t\cdot v_C.$$
 As before the interval $[v_C,v_{C,r}]$ is contained in $\Delta_C$, as well as its subset $$\conv(v_C,\Delta_{C,r})\subset \conv(v_C,\Delta_{C,r+1})\subset \ldots\subset \conv(v_C,\Delta_{C,k})$$\noindent
 (since $\Delta_{C,r})\subset (v_C,\Delta_{C,r+1})\subset \ldots\subset \conv(\Delta_{C,k})$).
 The restriction of the homotopy equivalence $\Sigma_k\to \Sigma'_k$ defines homotopy equivalences between $\Sigma_i$ and $\Sigma'_i$, for $i\leq k$.

\end{proof}

Theorem \ref{8} can be used, for instance,  for studying fibers. In this situation the codimension one assumption is usually not satisfied. If the singular locus is not a finite set of points then the exceptional fibers are usually (generically) of lower dimension. That is why we generalize Theorems \ref{7}, and \ref{8}, \ref{9} dropping the codimension one assumption.

Given a union $\bigcup E_i$ of nonsingular closed subvarieties $E_i$,
we will say that it  has simple normal crossings (SNC) if around every
point $E$ is  locally
analytically equivalent to a union of intersections of coordinate hyperplanes.
 We can define the dual complex
$\Sigma_E$ as above.
We assume that $E_i$ are maximal components and they are assigned vertices $p_i$. The  simplices  $\Delta_{p_{i_1},\ldots,p_{i_k}}$ correspond to the components of 
 $E_{i_1},\cap \ldots\cap E_{i_k}$ as before.

\begin{thm} \label{10} Suppose that  $\phi :X\dashrightarrow  X'$ is a birational map of two nonsingular complete varieties, such that its restriction $\phi|_{U}: U\dashrightarrow  U'$ to   open sets $U\subset X$ and $U'\subset X'$ is proper. 
Assume that the boundary sets  $E=X\setminus U$ and $E'=X'\setminus U'$ are unions  of the SNC components.
Then the dual complexes $\Sigma_E$ and $\Sigma_{E'}$ are homotopically equivalent.
\end{thm}
\begin{proof} The theorem follows from the  Lemma:

\end{proof}
\begin{lemma} Let  $E=X\setminus U$  be an SNC set in a nonsingular
  variety $X$  with components $E=\bigcup E_i$. There exists a sequence of blow-ups  $X'\to X$ with centres which are intersections of the strict transforms of the maximal components of $E$ such that $D=X'\setminus U$ is a SNC divisor, and there is a homotopy equivalence $\Sigma_D\to \Sigma_E$.
\end{lemma}

\begin{proof} Consider the smallest possible intersection component $E(\Delta_0)$  (i.e. intersection of some $E_i$) corresponding to the maximal face $\Delta_0\in \Sigma_E$ of a certain dimension $k$. Consider the subcomplex $\Sigma_{E,\Delta_0}$ consisting of all faces $\Delta$ of $\Sigma_E$ such that the corresponding intersection components $E(\Delta)$ are not contained in $E(\Delta_0)$. Note that $\Sigma_{E,\Delta_0}$ intersects with $\Delta_0$ along a certain subcomplex $\Sigma_{\Delta_0}$. The set $\Sigma_{E,\Delta_0}^c:=\Sigma\setminus\Sigma_{E,\Delta_0}$ consists of simplices $\Delta$ for which $E(\Delta)$ are contained in $E(\Delta_0)$. By minimality condition for $E(\Delta_0)$, this  means $E(\Delta)=E(\Delta_0)$.
The simplices in $\Sigma_{E,\Delta_0}^c$ are the faces of $\Delta_0$ with the property $E(\Delta)=E(\Delta_0)$.

 After blow-up along the centre $C=E(\Delta_0)$, the set of maximal components will be enlarged by adding the exceptional divisor $D_0$ to the previous set $\{E_i\}$. 
 
 (If  $E(\Delta_0)$ is not a maximal component  the set of ``old'' maximal components $E_i$ remains the same  after the blow-up. If $E(\Delta_0)=E_i$ is  a maximal component then the corresponding strict transform of $E_i$ disappear.).

The complex $\Sigma_{E,\Delta_0}$ will remain the same after blow-up. The components defined by the afces in $\Sigma_{E,\Delta_0}^c$ will disappear.
The simplices of  $\Sigma_{\Delta_0}=\Sigma_{E,\Delta_0}\cap \Delta_0$ will be joined with a new vertex $v_0$ corresponding to the exceptional divisor $D_0$.
The corresponding components $E(\Delta)$ intersect the centre but do not contain it. Thus their strict transform will intersect $D_0$.
 The simplices will form a complex $v_0*\Sigma_{\Delta_0}$. 
 
 Then,
the new complex is $$\Sigma_{E,1}:= \Sigma_{E,\Delta_0}\cup v_0*\Sigma_{\Delta_0}.$$
 There is a deformation retraction of $\Delta_0\to v_0*\Sigma_{\Delta_0}$ which is identity 
 on $ v_0*\Sigma_{\Delta_0}$. Consider the stellar subdivision $v_0\cdot \Delta_0$ of $\Delta_0$ at its  barycentre $v_0$.
 
 Let $\{v_i\}_{i\in I_{k-1}}$ denotes the set  of all the barycentres of all the $k-1$ dimensional faces in $(\Sigma_{E,\Delta_0}^c)\subset\partial(\Delta)$. Note that $\{v_i\}_{i\in I_{k-1}}$ lie in pairwise distinct simplices, and the star subdivisions at $v_i$ commute.
Take the stellar subdivisions $$\{v_i\}_{i\in I_{k-1}}\cdot v_0\cdot \Delta_0$$ of $v_0\cdot \Delta_0$ at  the centres $\{v_i\}_{i\in I_{k-1}}$. The linear map transform all $v_i\mapsto v_0$, and is identity on all other vertices of $\{v_i\}_{i\in I_{k-1}}\cdot v_0\cdot \Delta_0$. It  defines the homotopic retraction of $$v_0\cdot \Delta_0=v_0* \partial(\Delta_0)\simeq  v_0*(\{v_i\}_{i\in I_{k-1}}\cdot \partial(\Delta_0))\to v_0* ((\Sigma_{E,\Delta_0}^c)^{k-2}\cup \Sigma_{\Delta_0}),$$
\noindent where 
$\partial(\Delta_0)$ consists of proper faces of $\Delta_0$, and
$(\Sigma_{E,\Delta_0}^c)^{k-2}$ consists of all faces of $\Sigma_{E,\Delta_0}^c$ of dimension $\leq k-2$.
Then take stellar subdivision of  $v_0* ((\Sigma_{E,\Delta_0}^c)^{k-2})$ at  the centres $\{v_i\}_{i\in I_{k-2}}\subset \partial(\Delta_0)$ of all the $k-2$ dimensional faces $(\Sigma_{E,\Delta_0}^c)\subset\Delta$. The linear map $v_i\to v_0$ defines the retraction
$$ v_0* ((\Sigma_{E,\Delta_0}^c)^{k-2}\cup \Sigma_{\Delta_0})\simeq v_0* (\{v_i\}_{i\in I_{k-2}}\cdot (\Sigma_{E,\Delta_0}^c)^{k-2}\cup \Sigma_{\Delta_0})\to v_0* ((\Sigma_{E,\Delta_0}^c)^{k-3}\cup \Sigma_{\Delta_0}).$$
By continuing this process we  get the retraction $\Delta_0\to v_0*\Sigma_{\Delta_0}$ which extends to the homotopic retraction
$$\Sigma_{E}= \Sigma_{E,\Delta_0}\cup \Delta_0\to \Sigma_{E,1}= \Sigma_{E,\Delta_0}\cup v_0*\Sigma_{\Delta_0}.$$

The complex  $\overline{\Sigma}_{E,1}:=\Sigma_{E,\Delta_0}$ is a subcomplex  $\Sigma_{E,1}$ which corresponds to the strict transforms of the intersection components of $E$.
 The subcomplex  $v_0*\Sigma_{\Delta_0}$  in   the complex 
 $\Sigma_{E,1}= \overline{\Sigma}_{E,1}\cup v_0*\Sigma_{\Delta_0}$, corresponds to the intersections of $D$ with the   components in $\Sigma_{\Delta_0}\subset \overline{\Sigma}_{E,1}$.

Suppose that the exceptional divisor $D_0$ contains a strict transform of an intersection component $E(\Delta)$. Then  the center of blow-up $E(\Delta_0)$ on $X$ contains $E(\Delta)$. By minimality $E(\Delta)=E(\Delta_0)$.
 But after the blow-up at $C=E(\Delta_0)$, the strict transforms of the  components intersecting at $E(\Delta_0)$ and being normal crossings do not intersect anymore on $X_1$.
 
  Thus  $D_0$ does not contain any strict transforms of the intersection components of $E$.
Since   $D_0$  has SNC with these components it yields that  $D_0$ and the  strict transforms of the intersection components of $E$ are described locally by different coordinates in a certain coordinate system. In particular, for any $\Delta,\Delta' \in \overline{\Sigma}_{E,1}$  we have the implication:
$$\emptyset\neq E(\Delta)\cap D \subset E(\Delta') \quad \Rightarrow \quad E(\Delta) \subset E(\Delta')\quad\quad\quad\quad \hfill{(*)}$$

 Let  $E(\Delta_1)$, where $\Delta_1\in \overline{\Sigma}_{E_1}$ be the smallest possible intersection component (not contained in $D$) corresponding to the maximal face $\Delta_1\in \
 \overline{\Sigma}_{E_1}$ of a certain dimesion $k_1$.
If $E(\Delta_1)$ intersects $D$  the construction and reasoning are the same as before and we get  (with the relevant notation):
$$\Sigma_{E,1}= \Sigma_{E,1,\Delta_1}\cup \Delta_1\to \Sigma_{E,2}= \Sigma_{E,1,\Delta_1}\cup v_1*\Sigma_{\Delta_1}= \overline{\Sigma}_{E_2}\cup v_0*\Sigma_{\Delta_0}\cup v_1*\Sigma_{\Delta_1}.$$

Assume that $E(\Delta_1)$ intersects $D$. Consider the subcomplex $\Sigma_{E_1,\Delta_1}$ consisting of all faces $\Delta$ of $\Sigma_{E,1}$ such that the corresponding intersection component $E(\Delta)$ are not contained in $E(\Delta_1)$. 

Note that $\Sigma_{E,1,\Delta_1}^c:=\Sigma_{E,1}\setminus\Sigma_{E,1,\Delta_1}$ consists of simplices $\Delta$ for which $E(\Delta)$ are  contained in $E(\Delta_1)$. By minimality condition, this  means $E(\Delta)=E(\Delta_1)$ or $E(\Delta)=E(\Delta_1)\cap D$. In particular, all the simplices of $\Sigma_{E,1,\Delta_1}^c$ are faces of $\conv(\Delta_1,v_1)$.
Also, by (*) a face $\Delta$ of $\Delta_1$ is in $\Sigma_{E,1,\Delta_1}^c$ if $\conv(\Delta,v_0)$ is in $\Sigma_{E, 1,\Delta_1}^c$ (and vice versa).

Consider the intersection  $$\Sigma_{\Delta_1}:=\Sigma_{E,1,\Delta_1}\cap\Delta_1.$$
By  above $\Sigma_{E, 1,\Delta_1}$ intersects with $\conv(\Delta_1,v_0)$  along a subcomplex
  $$\Sigma_{E,1,\Delta_1}\cap\conv(\Delta_1,v_0)=v_0* \Sigma_{\Delta_1}.$$  
 After blow-up at $C=E(\Delta_1)$  we construct the exceptional divisor $D_1$ corresponding
 to the new vertex $v_1$ which is a barycentre of $\Delta_1$. The subcomplex 
 $\Sigma_{E_1,\Delta_1}$ remains unchanged.
 Consider the stellar subdivision $v_1\cdot \Delta_1$ of $\Delta_1$ at its  barycentre $v_1$.

 There is a deformation retraction of $$v_0* \Delta_1:=\conv(\Delta_1,v_0)\simeq v_1\cdot(v_0*\Delta_1) \to v_1* v_0*\Sigma_{\Delta_1}$$ which is identity 
 on $ v_1*\Sigma_{\Delta_1}$. Take the stellar subdivisions $$\{v_i\}_{i\in I_{k_1-1}}\cdot v_1\cdot (v_0*\Delta_1)$$ \noindent  of $v_1\cdot (v_0*\Delta_1)$ at  the barycentres  $\{v_i\}_{i\in I_{k_1-1}}\subset \partial(\Delta_1)$ of all the $k_1-1$ dimensional faces $\Sigma_{E_1,\Delta_1}^c\cap \partial(\Delta_1)\subset \partial(\Delta_1)$. 

Construct the linear map which transforms $v_i$ to $v_1$, and which is identity on all other vertices of  $\{v_i\}_{i\in I_{k_1-1}}\cdot v_1\cdot \conv(\Delta_1,v_0)$. It  defines the homotopic retraction  $$ v_0*( v_1\cdot\Delta_1)=v_1* v_0* \partial(\Delta_1)\to v_1* v_0*(((\Sigma_{E,\Delta_1}^c)^{k_1-2}\cap \partial(\Delta_1)) \cup \Sigma_{\Delta_1}),$$
\noindent where $(\Sigma_{E,\Delta_1}^c)^{k_1-2}$ consists of all faces of $\Sigma_{E,\Delta_1}^c$ of dimension $\leq k_1-2$.
Then take stellar subdivision of at  the centres $\{v_i\}_{i\in I_{k_1-2}}\cdot v_1\cdot (v_0*\Delta_1)$ at the barycentres of all the $k_1-2$ dimensional faces $(\Sigma_{E,\Delta_1}^c)\cap \partial(\Delta_1)$. The linear map $v_i\to v_1$  defines the retraction
$$ v_1* v_0* (((\Sigma_{E,\Delta_1}^c)^{k_1-2}\cap  \partial(\Delta_1)) \cup \Sigma_{\Delta_1})\to v_1*v_0* (((\Sigma_{E,\Delta_1}^c)^{k_1-3}\cap \partial(\Delta_1))\cup \Sigma_{\Delta_1}).$$
By continuing this process,
 we  get the retraction $v_0*\Delta_1\to v_1* v_0*\Sigma_{\Delta_1}$ which extends to the deformation retraction
$$\Sigma_{E,1}= \Sigma_{E_1,\Delta_1}\cup v_0*\Delta_1=\Sigma_{E_1,\Delta_1}\cup v_0*v_1*\partial(\Delta_1) \to  \Sigma_{E,2}= \Sigma_{E,1,\Delta_1}\cup v_0* v_1*\Sigma_{\Delta_1}.$$
Here $ \overline{\Sigma}_{E,2}:= \Sigma_{E,1,\Delta_1}$ is the subcomplex of $\Sigma_{E,1}$ and $\Sigma_{E,2}$ corresponding to the nonempty strict transforms  of of the $E$-components.

We construct new divisors $D_1$ and $D_2$ intersecting transversally with the strict transforms of the $E$-components. The subcomplex $\overline{\Sigma}_{E,2}$ of $\Sigma_{E,2}$  defined by $E$-components is contained as a proper subset in $\overline{\Sigma}_{E,1}\subsetneq \Sigma_{E}$.
By continuing this algorithm, we construct a sequence of blow-ups such that the subcomplex $\overline{\Sigma}_{E,k}$ of $\Sigma_{E,k}$ is empty. The resulting boundary set becomes divisorial. Note that  we eliminate one by one all the simplices of $\Sigma_E$. The vertices of $\Sigma_E$ will be eliminated upon the blowing up the corresponding maximal components at the very end of the process.
\end{proof}

The lemma above can be extended to the multiple subvarieties case.

\begin{cor}\label{12} Let $E=X\setminus U $   be   a union  of the closed SNC components 
$E=\bigcup E_i$ on a nonsingular variety $X$. Let
$E^1\subset E^2\subset\ldots\subset E^r=E$ be the filtration of unions of some components $E_i$. There exists a sequence of blow-ups  $X'\to X$ with centres which are intersections of the strict transforms of maximal components of $E^r$ such that the inverse images $D^j$ of $E^j$ is a SNC divisor, for $j=1,\ldots,r$ and there is a homotopy equivalence of topological spaces $\Sigma_{D^1}\subset \Sigma_{D^2}\subset\ldots\subset \Sigma_{D^r}$ and $\Sigma_{E^1}\subset \Sigma_{E^2}\subset\ldots\subset \Sigma_{E^r}$ .
\end{cor}

\begin{proof} This is pretty much identical to the proof of the previous Lemma.
Consider the smallest possible intersection component $E_0$ of $E=E^r$.  Assume $E_0$ is a an intersection component of $E^j$ but is not an intersection component of $E^i$ for $i<j$.
For any $i\geq j$ consider the maximal face $\Delta_{0,i}\in \Sigma_{E^i}$ corresponding to $E_0$. In particular  $E(\Delta_{0,i})=E_0$ for $i\geq j$, and $$\Delta_{0,j}\subseteq \Delta_{0,j+1}\subseteq \ldots\subseteq \Delta_{0,k}$$
are face inclusions.
Let $v_0\in \Delta_{0,j}$ be its barycentre.
Observe that  for $i<j$ the blow-up of $C=E_0$ does not change $\Sigma_{E^i}$. 

Let $i\geq j$.  Consider the complex  $\Sigma_{E^i,\Delta_{0,i}}$ consisting of  all faces $\Delta$ of $\Sigma_{E^i}$ such that the corresponding intersection components $E(\Delta)$ are not contained in $E(\Delta_0)$.  As before the simplices in $\Sigma_{E^i,\Delta_{0,i}}^c:=\Sigma_{E^i}\setminus \Sigma_{E^i,\Delta_{0,i}}$ are the faces of $\Delta_{0,i}\in \Sigma_{E^i}$ with the property $E(\Delta)=E(\Delta_{0,i})$. 
It follows that $$\Sigma_{E^{i+1},\Delta_{0,i+1}}^c\setminus \Sigma_{E^i,\Delta_{0,i}}^c\subset \Delta_{0,i+1}\setminus \Delta_{0,i}.$$

The complexes  $\Sigma_{E^i,\Delta_{0,i}}$ will remain the same after blow-up at $E_0$.
The components defined by the faces in $\Sigma_{E^i,\Delta_{0,i}}^c$ will disappear.
The simplices of  $\Sigma_{\Delta_{0,i}}=\Sigma_{E^i,\Delta_{0,i}}\cap \Delta_{0,i}$ will be joined with a new vertex $v_0$ corresponding to the exceptional divisor $D_0$.
 They will form a complex $v_0*\Sigma_{\Delta_{0,i}}$. 

Thus for $i\geq j$ the new complex has the form 
$\Sigma_{E'^i}=\Sigma_{E^i,\Delta_{0,i}}\cup \quad v_0* \Sigma_{\Delta_{0,i}}$. It  is a subcomplex  of $v_0\cdot\Sigma_{E'^i}$ with deformation retraction defined on

$$v_0* \Delta_{0,j}\subseteq v_0\cdot\Delta_{0,j+1}\subseteq \ldots\subseteq v_0\cdot\Delta_{0,r}$$
to

$$v_0*\Sigma_{\Delta_{0,j}}\subseteq v_0*\Sigma_{\Delta_{0,j+1}}\subseteq \ldots\subseteq v_0*\Sigma_{\Delta_{0,r}},$$

The construction of the deformation retraction is almost identical as before.
Note that $\Sigma_{\Delta_{0,i}}\cap \Delta_{0, i'}=\Sigma_{\Delta_{0,i'}}$ for $i\geq i'\geq j$.
First we define the retraction $$v_0\cdot\Delta_{0,r}\to v_0*(\Sigma_{\Delta_{0,r}}\cup \Delta_{0,r-1}),$$ \noindent eliminating all the faces in $\Sigma_{E^r,\Delta_{0,r}}^c\setminus \Delta_{0,r-1}$, and  identical on $\Delta_{r-1}$. (We use the same technique as in the proof of the previous Lemma.)

Then we retract  $$v_0*\Delta_{0,r-1}\to v_0*(\Sigma_{\Delta_{0,r-1}}\cup \Delta_{0,r-2}).$$ It defines retraction  $$v_0*(\Sigma_{\Delta_{0,r}}\cup \Delta_{0,r-1}) \to v_0*(\Sigma_{\Delta_{0,r}}\cup \Delta_{0,r-2}),$$
and  $$v_0*\Delta_{0,r} \to v_0*(\Sigma_{\Delta_{0,r}}\cup \Delta_{0,r-2}),$$

We continue this process down for  $i=r,r-1,\ldots, j$ to get desired compatible retractions 
$v_0*\Delta_{0,i}\to v_0*(\Sigma_{\Delta_{0,i}})$.

Finally   extending the above retractions  by the identity yields  the retraction of 
 $$\Sigma_{E^1}\subset \Sigma_{E^2}\subset\ldots\subset \Sigma_{E^r}$$
to
$$\Sigma_{E'^1}\subset \Sigma_{E'^2}\subset\ldots\subset \Sigma_{E'^r}.$$
 The rest of the proof is the same.
\end{proof}

\begin{cor}\label{11}Let $E=X\setminus U $ (respectively  $E'= X'\setminus U' $ ) be  a union  of the closed SNC components 
$E=\bigcup E_i$  (resp. $E'=\bigcup E'_i$) on  a nonsingular variety $X$ (resp.  $X'$). Let
$E^1\subset E^2\subset\ldots\subset E^r=E$ be the union of some components $E_i$, 
(resp.  $E'^1\subset E'^2\subset\ldots\subset E'^r=E'$ be the union of some components $E'_i$.)

Assume that there exists a proper birational map $\phi: X\dashrightarrow X'$ such that the restrictions
$\phi_{X\setminus E^j}: X\setminus E^j\dashrightarrow X'\setminus E'^j$ are proper for $j=1,\ldots, k$.

Then the corresponding sequences of topological spaces $\Sigma_{E^1}\subset\Sigma_{E^2}\subset\ldots \subset \Sigma_{E^r}$ and $\Sigma_{E'^1}\subset\Sigma_{E'^2}\subset\ldots \subset \Sigma_{E'^r}$  are homotopically equivalent.
\end{cor}

\begin{proof} Follows from the Corollary \ref{12} and Theorem \ref{9}.
\end{proof}

\end{document}